\documentclass[12pt,twoside]{article}
\usepackage{etex}
\usepackage[dvipsnames]{xcolor} 
\usepackage{color}
\usepackage{booktabs}
\usepackage{amsthm}
\usepackage{amsfonts,amssymb,amsxtra,url,float} 
\allowdisplaybreaks[4] 
\usepackage[colorlinks,
  linkcolor=magenta, %
  anchorcolor=Periwinkle,
  citecolor=violet,
  urlcolor=blue
  ]{hyperref} 
\usepackage{enumitem}
\usepackage{geometry} 
\usepackage{rotating} 
\usepackage{lscape} 
\usepackage{multirow}
\usepackage{graphicx} 
\usepackage{subfigure} 
\usepackage{tikz}
\usepackage{pgfplots}
\usepackage{tikz-3dplot}
\usetikzlibrary{patterns}
\usetikzlibrary{3d,calc}
\usetikzlibrary{decorations.pathreplacing,decorations.markings}
 \tikzset{
  on each segment/.style={
    decorate,
    decoration={
      show path construction,
      moveto code={},
      lineto code={
        \path [#1]
        (\tikzinputsegmentfirst) -- (\tikzinputsegmentlast);
      },
      curveto code={
        \path [#1] (\tikzinputsegmentfirst)
        .. controls
        (\tikzinputsegmentsupporta) and (\tikzinputsegmentsupportb)
        ..
        (\tikzinputsegmentlast);
      },
      closepath code={
        \path [#1]
        (\tikzinputsegmentfirst) -- (\tikzinputsegmentlast);
      },
    },
  },
  mid arrow/.style={postaction={decorate,decoration={
        markings,
        mark=at position 0.6 with {\arrow[#1]{stealth}} 
      }}},
}
\usetikzlibrary{arrows}
\usetikzlibrary{trees}

\usetikzlibrary{matrix}
\usetikzlibrary{patterns}
\usetikzlibrary{shadings} 
\usepackage{fancyhdr} 
\def\headertitle{Module-valued ordinary differential equations and structure of solution spaces}
\def\fstpage{1} 
\def\page{$\begin{matrix} {\color{white}0} \\ \thepage \end{matrix}$} 

\usepackage[all]{xy} 
\usepackage{dsfont} 
\usepackage{cite}
\usepackage{mathrsfs} 
\numberwithin{figure}{section}
\usepackage{marginnote} 
\usepackage{graphicx} 
\usepackage{multicol} 

\usepackage{enumitem}
\setenumerate[1]{itemsep=0pt,partopsep=0pt,parsep=\parskip,topsep=3pt}
\setitemize[1]{itemsep=0pt,partopsep=0pt,parsep=\parskip,topsep=3pt}
\setdescription{itemsep=0pt,partopsep=0pt,parsep=\parskip,topsep=3pt}
\setlist[itemize]{leftmargin=35pt}
\setlist[enumerate]{leftmargin=35pt}
\usepackage{changepage} 
\newcommand{\checks}[1]{{\color{black}{#1}}} 

\newtheorem{theorem}{Theorem}[section]
\newtheorem{lemma}[theorem]{Lemma}
\newtheorem{corollary}[theorem]{Corollary}
\newtheorem{main theorem}[theorem]{Main Theorem}
\newtheorem{proposition}[theorem]{Proposition}
\newtheorem{definition}[theorem]{Definition}
\newtheorem{construction}[theorem]{Construction}
\newtheorem{remark}[theorem]{Remark}
\newtheorem{example}[theorem]{Example}
\newtheorem{notation}[theorem]{Notation}

\newtheorem{assumption}[theorem]{Assumption}

\usetikzlibrary{arrows}

\numberwithin{equation}{section}

\def\orcid{
\begin{tikzpicture}[baseline=-1mm]
\filldraw[Green!35] (0,0) circle (5pt);
\filldraw[white] (0,0) node{\tiny\textbf{iD}};
\end{tikzpicture}
}

\def\RR{\mathbb{R}} 
\def\CC{\mathbb{C}} 
\def\II{\mathbb{I}}

\newcommand{\Ker}{\operatorname{Ker}}

\newcommand{\modcat}{\mathsf{mod}}

\newcommand{\kk}{\Bbbk} 
\newcommand{\Q}{\mathcal{Q}} 
\def\I{\mathcal{I}}

\newcommand{\e}{\varepsilon}
\newcommand{\Hom}{\mathrm{Hom}} %
\newcommand{\End}{\mathrm{End}} %

\newcommand{\op}{\mathrm{op}}
\newcommand{\w}[1]{\widehat{#1}}

\def\bfS{\mathbf{S}}

\def\id{\mathbf{1}}
\def\ident{\mathrm{id}}
\def\frakI{\mathfrak{I}}

\def\emb{\mathrm{emb}}
\def\im{\mathrm{Im}}
\def\calH{\mathcal{H}}

\def\compos{\ \lower-0.2ex\hbox{\tikz\draw (0pt, 0pt) circle (.1em);} \ }

\newcommand{\defines}[1]{{\it\color{black}{#1}}}
\title{\bf Module-valued ordinary differential equations and structure of solution spaces
}

\vspace{5mm}

\author{
Shengda Liu$^{\ref{Author1}, \href{https://orcid.org/0000-0003-1382-4212}{\orcid}\ref{orcid1}}$ ~
Yu-Zhe Liu$^{\ref{Author2}, \href{https://orcid.org/0009-0005-1110-386X}{\orcid}\ref{orcid2},~\ref{CorrespondingAuthor}}$ ~
Keyu Tao$^{\ref{Author3}, \href{https://orcid.org/0009-0005-1110-386X}{\orcid}\ref{orcid3}}$
}
\date{ }
\begin{document}






\maketitle

\begin{enumerate}[label=\textbf{\color{red}\arabic*}] \footnotesize
  \item
    \begin{center}
      The Key Laboratory of Cognition and Decision Intelligence for Complex Systems,
      Institute of Automation, Chinese Academy of Sciences, Beijing 100190, China;

      E-mail:  \url{thinksheng@foxmail.com}
    \end{center} \label{Author1}
  \item
    \begin{center}
      School of Mathematics and Statistics, Guizhou University,
      Guiyang 550025, Guizhou, China;

      E-mail:  \url{liuyz@gzu.edu.cn} / \url{yzliu3@163.com}
    \end{center} \label{Author2}
  \item
    \begin{center}
      School of Mathematics, Nanjing University,
      Nanjing 210008, Jiangsu, China;

      E-mail:  \url{tkytay@163.com}
    \end{center} \label{Author3}
\end{enumerate}
\vspace{2mm}
\begin{enumerate}[label=\textbf{\color{red}$\dag$}]
  \item \footnotesize
    \begin{center}
      Corresponding author
    \end{center} \label{CorrespondingAuthor}
\end{enumerate}
\vspace{2mm}
\begin{enumerate} \footnotesize
  \item[{\orcid}] \centering  ORCID: \href{https://orcid.org/0000-0003-1382-4212}{0000-0003-1382-4212}
      \label{orcid1}
  \item[{\orcid}] \centering  ORCID: \href{https://orcid.org/0009-0005-1110-386X}{0009-0005-1110-386X}
      \label{orcid2}
  \item[{\orcid}] \centering  ORCID: \href{https://orcid.org/0009-0004-6054-6897}{0009-0004-6054-6897}
      \label{orcid3}
\end{enumerate}

\vspace{1mm}

\begin{adjustwidth}{1cm}{1cm}
  \noindent \footnotesize
  \textbf{Abstract}: We define and study homogeneous linear ordinary differential equations for functions valued in a Banach module $V$ over a finite-dimensional $\Bbbk$-algebra $\mathit{\Lambda}$ by using the tensor of Banach modules. Furthermore, we show that the solution space is a finitely generated $\mathit{\Lambda}$-submodule.

\vspace{1mm}

     \label{2020MSC}

\vspace{1mm}

     \label{Keywords}
\end{adjustwidth}

\tableofcontents


\def\la{\langle} 
\def\ra{\rangle} 
\def\lala{\langle\!\langle}
\def\rara{\rangle\!\rangle}
\def\=<{\leqslant}
\def\>={\geqslant}
\def\vecd{\pmb{d}}
\def\vecg{\pmb{g}}
\def\spa{\mathrm{span}}
\def\Gal{\mathrm{Gal}}
\def\Mat{\mathbf{Mat}}
\def\bfE{\pmb{E}}
\def\op{\mathrm{op}}
\def\norm{\mathfrak{n}}
\def\bigboxplus{\mathop{ \raisebox{-0.3em}{ \scalebox{1.75}{\text{\(\boxplus\)}} } }}

\def\scrN{\mathscr{N}\!or}
\def\scrA{\mathscr{A}}
\def\Pc{\mathds{P}}

\def\invlim{\underleftarrow{\lim}}
\def\dirlim{\underrightarrow{\lim}}

\def\dimA{d_A}
\def\dimB{d_B}
\def\homo{\varsigma}
\def\itLamb{\mathit{\Lambda}}

\def\dd{\mathrm{d}}
\def\Bochner{\mathrm{B}}
\def\Lebesgue{\mathrm{L}}
\def\compos{{\begin{smallmatrix}\circ\end{smallmatrix}}}

\newcommand{\basis}[1]{\mathfrak{B}_{#1}}

\newcommand{\Sol}{\mathbf{Sol}}            
\newcommand{\LODE}{\mathbf{L}}
\newcommand{\EndL}{\End_{\itLamb}}         
\newcommand{\HomL}{\Hom_{\itLamb}}         
\newcommand{\Lip}{\mathrm{Lip}}            
\newcommand{\Dop}{D}                       
\newcommand{\TV}{\w{T}_V}                  
\newcommand{\DV}{\Dop_V}                   
\newcommand{\SV}{\bfS_{\homo}(\II_{\itLamb}, V)} 
\newcommand{\wSV}{\w{\bfS_{\homo}(\II_{\itLamb}, V)}} 
\newcommand{\wS}{\w{\bfS_{\homo}(\II_{\itLamb})}} 

\section{Introduction}\label{sec:intro}

The classical theory of ordinary differential equations (ODEs) is built over the field $\RR$ (or $\CC$):
the domain is an interval $[a,b] \subseteq \RR$, and the unknown takes values in a Banach space.
While this framework encompasses the vast majority of applications,
it does not account for the algebraic structure of the domain.
Replacing $\RR$ by a finite-dimensional $\kk$-algebra $\itLamb$ raises the question of how
the representation theory of $\itLamb$ governs the structure of differential equations and their solutions.

Categorical or algebraic approaches to calculus have developed along several lines,
including integration \cite{Sag1965ana, CL2018int, CL2018Cart-int, Lei2023FA},
differentiations \cite{Lemay2019, Lemay2021exp, ML2021diff, Lemay2023, IL2023ana},
\checks{and} equations \cite{Kashiwara1995, CCL2021, HTT2008}.
Note that the Rota-Baxter algebra \cite{Bax1960,Rot1969-1,Rot1969-2} provides another algebraic description of integration,
but it is different from the categorification of integration given by the \checks{categorical approaches} provided in \cite{Lei2023FA,LLHZ2025}.
Several important works deserve to be mentioned in this context. For example,
Ehrhard and Regnier \cite{ER2003} introduced differential operations in linear logic;
Blute, Cockett and Seely \cite{BCS2006diff} developed the theory of differential categories;
Cockett, Cruttwell and Lemay \cite{CCL2021} defined differential equations in tangent categories via curve objects;
and Lemay \cite{Lemay2021exp} studied exponential functions in Cartesian differential categories.
These frameworks operate at a high level of abstraction and have not been instantiated on finite-dimensional algebras.
We note that the algebra $\itLamb$ in our setting is finite-dimensional,
and its module category $\modcat(\itLamb)$ possesses rich representation-theoretic structure,
which distinguishes the present work from $D$-module theory \cite{Kashiwara1995, HTT2008}
where the differential operator ring is infinite-dimensional.
On the other hand, Leinster \cite{Lei2023FA} obtained a categorical derivation of the Lebesgue integral
as the unique morphism from the initial object of a category $\scrA^p$.
Building on Leinster's work, the authors and collaborators extended the framework to finite-dimensional algebras.
In \cite{LLHZ2025}, the normed module category $\scrN^p$ and the integral Banach module category $\scrA^p$
of a finite-dimensional $\kk$-algebra $\itLamb$ were defined,
the Lebesgue integral was realized as the unique morphism from the initial object,
and the differential operator $\Dop$ was identified as a morphism in $\scrA^1$ between two distinct objects,
revealing a categorical asymmetry between integration and differentiation.
Subsequently, \cite{GLLWpre} established transformation formulas between integrals on different algebras,
and \cite{LLL2025cat3} defined integrals with variable upper limits via integral partial ordered sets
and proved a categorical fundamental theorem of calculus.

Assume $\itLamb$ is a finite-dimensional algebra defined on a complete field $\kk$,
\checks{that} $\kk$ is an extension of $\RR$, and \checks{let} $V$ be a Banach $\itLamb$-module
(all $\itLamb$-modules in this paper are left $\itLamb$-modules).
In this paper, we define and study ordinary differential equations within this categorical calculus framework.
We construct $V$-valued function spaces as objects of $\scrA^p$ (Proposition \ref{prop:V-object}),
introduce the $V$-valued differential operator $\DV$,
the i.poset $\w{T}_{x_0,V}^x(f)$ of $f$,
and formulate categorical ODEs for $V$-valued functions (Definition \ref{def:ODE}).
The main results are as follows.

\begin{theorem} \label{thm:260430}
Under the generalized L-condition {\rm(}see \ref{L1}--\ref{L6} in this paper{\rm)},
let $\itLamb$ be a finite-dimensional algebra, $V$ be a finitely generated Banach $\itLamb$-module,
and $\w{T}_{x_0,V}^{x}(\DV f) = f(x)-f(x_0)$
holds for all absolute continuous functions $f\in\wSV$ and $x_0\preceq x$ {\rm(}see Assumption \ref{assump}{\rm)}\checks{.}
\checks{Then} the $V$-valued initial value problem
\begin{align}\label{eq in main thm}
  \DV f(x) = G(x, f(x)), ~ f(x_0)=f_0
\end{align}
satisfies the following statements.
\begin{enumerate}[label=\text{\rm(\arabic*)}]
\item {\rm \textbf{Equivalence} (Theorem \ref{thm:equivalence})}~
The equation {\rm(\ref{eq in main thm})} is equivalent to
\begin{center}
  $f(x)=f_0+\w{T}_{x_0,V}^x(G(\cdot,f(\cdot)))$.
\end{center}

\item {\rm\textbf{Existence and uniqueness} (Theorem \ref{thm:existence})}~
If $G$ is a function $[c,d]_{\itLamb} \times V \to V$ satisfying the Lipschitz condition with Lipschitz constant $\checks{K}$,
then {\rm(\ref{eq in main thm})} has a unique solution.

\item {\rm\textbf{Solution space structure} (Theorem \ref{thm:sol-module})}~
The solution space, say $\Sol$, of the equation $\DV f + A(x)\checks{\cdot_V} f = 0$ is a finitely generated $\itLamb$-module.
\end{enumerate}
\end{theorem}

Furthermore, we have the following corollary which describes the solution space $\Sol$ of $\DV f + A(x)\checks{\cdot_V} f = 0$.

\begin{corollary}[{\rm Corollary \ref{coro:sol-module}}]
\checks{Keeping} the notations from Theorem \ref{thm:260430}.
Then there is an isomorphism of $\itLamb$-modules
\[\Sol \cong V = \bigoplus_{i\in I} V_i \]
{\rm(}where $\{V_i\}_{i\in I}$ is the family running over all indecomposable direct summands of $V${\rm)},
which admits that each solution of {\rm(\ref{eq in main thm})} can be described by using the generators of $V$.
\end{corollary}

The above result is a corollary of Theorem \ref{thm:260430},
which can be proved by using the Krull--Schmidt decomposition of finitely generated $\itLamb$-module
\checks{\cite[Chap.~I, Theorem 4.10]{ASS2006}}.

\checks{This} paper is organized as follows.
Section \ref{sec:prelim} recalls the necessary background from \cite{LLHZ2025, GLLWpre, LLL2025cat3}.
Section \ref{sec:V-ODE} introduces $V$-valued function spaces and defines categorical ODEs.
Section \ref{sec:equiv} proves the equivalence of the differential and integral equation forms.
Section \ref{sec:exist} establishes existence and uniqueness, and analyzes the $\itLamb$-module structure of the solution space.

\section{Preliminaries}\label{sec:prelim}

This section reviews the basic definitions and key results from the earlier works \cite{LLHZ2025, GLLWpre, LLL2025cat3},
laying the foundation for the developments in subsequent sections.
We assume the reader is familiar with the basic language of category theory (see \cite{Eil1945, Mac1998}) and the representation theory of finite-dimensional algebras (see \cite{ASS2006}).

\subsection{Normed modules over finite-dimensional algebras}\label{subsec:algebra}

Let $\kk$ be a complete field, i.e., every Cauchy sequence in $\kk$ converges,
and suppose $\kk$ contains a totally ordered subset $[r,s]_{\kk}$, where $r, s\in \kk$ satisfy $r \preceq s$.
Typical examples are $\kk = \RR$ or $\CC$.

\begin{definition}[{\!\!\cite[Section 2]{LLHZ2025}}]\label{def:fd-algebra} \rm
Let $\itLamb$ be a finite-dimensional $\kk$-algebra, i.e., $\dim_{\kk}\itLamb = n < \infty$ as a $\kk$-linear space.
Fix a $\kk$-basis $B_{\itLamb} = \{b_1, b_2, \ldots, b_n\}$ of $\itLamb$.
\begin{enumerate}[label=\textrm{(\arabic*)}]
\item The \defines{integration region} is defined as
\begin{align}\label{eq:int-region}
  \II_{\itLamb} := [r,s]_{\kk}\cdot b_1 \times [r,s]_{\kk}\cdot b_2 \times \cdots \times [r,s]_{\kk}\cdot b_n
  \subseteq \itLamb.
\end{align}
\item The \defines{$p$-norm} on $\itLamb$ ($p \>= 1$) is defined by
\begin{align}\label{eq:norm-Lambda}
  \Big\|\sum_{i=1}^n k_i b_i\Big\|_p := \Big(\sum_{i=1}^n |k_i|^p\Big)^{1/p}, \quad k_i \in \kk.
\end{align}
\end{enumerate}
\end{definition}

Then $\itLamb$, as a $\kk$-linear space with $\Vert\cdot\Vert_p: \itLamb \to \RR^{\>=0}$, is a normed space.
Furthermore, if $\itLamb$ is complete, then it is a Banach algebra.
Given a measure $\mu:\kk\to\RR^{\>=0}$ defined on $\kk$\checks{,}
\checks{it} induces a measure $\mu_{\II_{\itLamb}}$ defined on the integration region $\II_{\itLamb}$, which satisfies
\begin{align}\label{eq:measure}
  \displaystyle \mu_{\II_{\itLamb}}(\II_{\itLamb}) := \prod_{i=1}^n \mu([r,s]_{\kk})
\end{align}
We will write $\mu_{\II_{\itLamb}}$ as $\mu$ for simplification \checks{when there is no confusion}.

\begin{definition}[{$\homo$-normed $\itLamb$-module, \cite[Section 4]{LLHZ2025}, \cite[Definition 3.2]{L2025normed}}]
\label{def:normed-module} \rm
Let $\homo: \itLamb \to \kk$ be a surjective $\kk$-algebra homomorphism.
A \defines{$\homo$-normed left $\itLamb$-module} (or a \defines{$\homo$-normed left $(\itLamb,\kk)$-bimodule}) is a pair $(M, \|\cdot\|)$, where $M$ is a left $\itLamb$-module
and $\|\cdot\|: M \to \RR^{\>= 0}$
is a norm such that $\|a \cdot m\| \=< |\homo(a)| \cdot \|m\|$ holds for all $a \in \itLamb$ and $m \in M$.
Furthermore, if $M$ is complete with respect to the norm $\|\cdot\|$ (i.e., every Cauchy sequence converges), then $(M, \|\cdot\|)$ is called a \defines{Banach left $\itLamb$-module}.
\end{definition}

For simplification, all $\itLamb$-modules in this paper are left $\itLamb$-modules.
Next, we recall that an important example of \checks{a} $\homo$-normed module
is the module of all elementary simple functions defined on $\II_{\itLamb}$.

\begin{definition}[{Elementary simple function, \cite[Section 3]{LLHZ2025}}]\label{def:simple-func} \rm
An \defines{elementary simple function} defined on $\II_{\itLamb}$ is a finite sum
\[ f = \sum_{i} k_i\id_{I_i}, \]
where $k_i \in \kk$; $I_i \subseteq \II_{\itLamb}$, and $\bigcup_{i}I_i = \II_{\itLamb}$;
and $\id_{I_i} := \begin{cases}
 1~(\in\kk), & x\in I_i; \\ 0, & x\notin I_i
\end{cases}$ is the indicator function (characteristic function) of $I_i$.
Furthermore, the set of all elementary simple functions is denoted by $\bfS_{\homo}(\II_{\itLamb})$.
\end{definition}

By \cite[Section 4]{LLHZ2025}, $\bfS_{\homo}(\II_{\itLamb})$ is a $\homo$-normed $\itLamb$-module,
whose $\itLamb$-module structure is defined by
\begin{align}\label{eq:module-action-S}
  a. f := (\homo(a) f: x\mapsto \homo(a)f(x)), \quad \forall a \in \itLamb, f \in \bfS_{\homo}(\II_{\itLamb}),
\end{align}
and whose norm is given by $\|f\|_1 := \sum_i |k_i|\mu(I_i)$, where $\mu$ is a measure on $\II_{\itLamb}$ satisfying (\ref{eq:measure}).
Its \defines{completion} is denoted by $\wS$, which is a Banach $\itLamb$-module.

\subsection{The categories $\scrN^p$ and $\scrA^p$}\label{subsec:cats}

Assume $p \>= 1$ is a real number.

\begin{definition}[{\!\!\cite[\checks{Section 5.1}]{LLHZ2025}}]
\label{def:Ap} \rm \checks{The} \defines{integral normed module category}
$\scrN^p := (\scrN^p, \homo:\itLamb \to \kk, \II_{\itLamb}, \mu)$ of $\itLamb$ is defined as follows:
\begin{enumerate}[label=\textrm{\checks{(\arabic*)}}]
  \item \defines{Objects} are triples $(N, v, \delta)$, where
    \begin{enumerate}[label=\textrm{\checks{(\roman*)}}]
      \item $N$ is a $\homo$-normed $\itLamb$-module;
      \item $v \in N$, $\|v\| \=< \mu(\II_{\itLamb})$;
      \item $\delta: N^{\oplus_p 2^n} \to N$ is a homomorphism of $\itLamb$-modules
        satisfying $\delta(v, v, \ldots, v) = v$,
            and $\delta$ commutes with inverse limits of Cauchy sequences.
    \end{enumerate}
  \item \defines{Morphisms} $(N, v, \delta) \to (N', v', \delta')$ are homomorphisms $\theta: N \to N'$
    of $\itLamb$-modules satisfying $\theta(v) = v'$ and such that the diagram
    \[
    \xymatrix{
      N^{\oplus_p 2^n} \ar[r]^{\delta} \ar[d]_{\theta^{\oplus 2^n}} & N \ar[d]^{\theta} \\
      N'^{\oplus_p 2^n} \ar[r]_{\delta'} & N'
    }
    \]
    commutes.
\end{enumerate}

Furthermore, \checks{the} \defines{integral Banach module category}
$\scrA^p:= (\scrA^p, \homo:\itLamb \to \kk, \II_{\itLamb}, \mu)$ of $\itLamb$
is the full subcategory of $\scrN^p$ of all objects $(N, v, \delta)$
where $N$ is a Banach $\itLamb$-module.
\end{definition}

\subsection{Initial objects and the integral morphism}\label{subsec:initial}

Let $\mathcal{C}$ be a category and $\mathcal{D}$ be a full subcategory of $\mathcal{C}$.
A $\mathcal{C}$-initial object in $\mathcal{D}$ is an object $X$ in $\mathcal{C}\backslash\mathcal{D}$
\checks{(that is, $X$ is an object of $\mathcal{C}$ that does not lie in the full subcategory $\mathcal{D}$)}
such that for each object $Y$ in $\mathcal{C}$, the Hom space $\Hom_{\mathcal{C}}(X,Y)$
contains only one morphism. Furthermore, the initial object of $\mathcal{C}$ is a $\mathcal{C}$-initial
since $\mathcal{C}$ is a trivial full subcategory of itself, cf. \cite[Page 216]{R1979}.
The following theorem is the key starting point of this paper.

\begin{theorem}[{\!\!\cite[Theorem 6.3]{LLHZ2025}}]\label{thm:initial}
The category $\scrN^p$ has a unique $\scrA^p$-initial object which is of the form
$(\bfS_\homo(\II_\itLamb), \id_{\II_\itLamb}, \gamma_\xi)$\checks{.}
\checks{Here} $\gamma_\xi: \bfS_\homo(\II_\itLamb)^{\oplus_p 2^n} \to \bfS_\homo(\II_\itLamb)$
is a special homomorphism which is called the juxtaposition map in \cite{Lei2023FA} and \cite{LLHZ2025}.
\end{theorem}

\begin{remark}\rm \label{rmk:initial}
In \cite{Lei2023FA}, the author considered a special case of $\scrA^p$
which \checks{was the} integral Banach module category of $\itLamb=\RR$.
In this case, $\II_{\itLamb}=[0,1]$, $\xi=\frac{1}{2}$,
and $\gamma_{\xi}=\gamma_{\frac{1}{2}}$ is the homomorphism induced by
\[ (f,g) \mapsto
\begin{cases}
  f(2x) & \displaystyle 0\=< x < \frac{1}{2}; \\
g(2x-1) & \displaystyle \frac{1}{2} \=< x \=< 1.
\end{cases} \]
Theorem \ref{thm:initial} \checks{induces} the following facts:
\begin{enumerate}[label=\textrm{(\arabic*)}]
  \item \label{thm:initial adm 1}
For any object $(N, v, \delta)$ in $\scrA^p$, there exists a unique $\scrN^p$-morphism
\[ T_{(N,v,\delta)}: (\bfS_\homo(\II_\itLamb), \id_{\II_\itLamb}, \gamma_\xi) \to (N, v, \delta), \]
such that it can be extended uniquely to \checks{a} homomorphism
\[ \w{T}_{(N,v,\delta)}: (\wS, \id_{\II_\itLamb}, \w{\gamma}_\xi) \to (N, v, \delta)\]
in $\Hom_{\scrA^p}((\wS, \id_{\II_\itLamb}, \w{\gamma}_\xi), (N, v, \delta))$.
  \item \label{thm:initial adm 2}
When the target object $(N, v, \delta)$ is taken to be $(\kk, \mu(\II_\itLamb), m)$
(where $m:\kk^{\oplus 2^n} \to \kk$ \checks{sends} each elements $\pmb{x}=(x_1,x_2,\ldots, x_{2^n})$ to
the weighted average of $x_1$, $x_2$, $\ldots$, $x_{2^n}$, see \cite{LLHZ2025} for details),
the homomorphism $\w{T}_{(\kk, \mu(\II_\itLamb), m)}$ yields \checks{an} \defines{integration}
in the case for $p=1$, \checks{as follows}:
\begin{align}\label{eq:integration}
 \forall f \in \wS, ~~  (\scrA^1)\int_{\II_\itLamb} f \dd\mu \ := \ \w{T}_{(\kk, \mu(\II_\itLamb), m)}(f).
\end{align}
  \item \label{thm:initial adm 3}
Under the Leinster's conditions (L-condition)
(cf. \cite[Example 3.9]{LLL2025cat3}),
(\ref{eq:integration}) is a Lebesgue integration (see \cite[Theorem 7.6]{LLHZ2025}),
and, in this case,
\begin{center}
  $\displaystyle \w{T}_{(\kk, \mu(\II_\itLamb), m)}(f) = (\Lebesgue) \int_{\II_A=[c,d]} f\dd\mu$
\end{center}
is the Lebesgue integral of $f$.
\end{enumerate}

\end{remark}

\subsection{The differential operator $\Dop$}\label{subsec:diff}

\begin{theorem}[{\!\!\cite[Theorem 9.5]{LLHZ2025}}]\label{thm:D}
Let $p = 1$, $\itLamb = \RR$ {\rm(}or $\CC${\rm)}, $\homo = \ident_\RR$, $\II_\itLamb = [0,1]$, and $\xi = \frac{1}{2}$. Then:
\begin{enumerate}[label=\text{\rm(\arabic*)}]
  \item \label{thm:D 1}
    There does not exist an $\scrA^1$-morphism $(\wS, \id, \w{\gamma}_{1/2}) \to (N, v, \delta)$ that maps $f$ to its weak derivative $\frac{\dd f}{\dd x}$.
  \item \label{thm:D 2}
    \checks{There} is a $\kk$-homomorphism $\Dop$ in
  \begin{align}\label{eq:D}
    \Hom_{\scrA^1}\big((\wS, \mathrm{id}, \w{\kappa}), (\wS, \id, \w{\gamma}_{\frac{1}{2}})\big)
  \end{align}
  such that it sends each $f\in \wS$ to the weak derivative of $f$, if $f$ lies in some Sobolev space $W^{1,1}~(\subseteq \wS)$.
  Here, $(\wS, \mathrm{id}, \w{\kappa})$ is another object in $\scrA^1$ different from $(\wS, \id, \w{\gamma}_{\frac{1}{2}})$.
\end{enumerate}
\end{theorem}

\begin{remark}\label{rmk:D-asymmetry} \rm \
\begin{enumerate}[label=\text{\rm(\arabic*)}]
  \item
We will not explain $(\wS, \mathrm{id}, \w{\kappa})$ in detail here because it is not necessary \checks{for this paper}. Readers only need to know that it is a special object.
  \item
Theorem \ref{thm:D} reveals the {asymmetry} between integration and differentiation at the categorical level:
the integral $\w{T}$ is a morphism emanating from the initial object $(\wS, \id, \w{\gamma}_{1/2})$,
while the differential $\Dop$ is a morphism {arriving at} that object (from $(\wS, \mathrm{id}, \w{\kappa})$).
This asymmetry is the starting point for the theory of differential equations in our framework.
\end{enumerate}
\end{remark}

\subsection{Integrals with variable upper limit}\label{subsec:FTC}


In \cite{LLL2025cat3}, the author introduced the integral partial ordered sets and \checks{used them} to describe the sequence of integrals.
In this sense, every integral with variable upper limit is a special sequence of integrals.

Given a continuous totally ordered subset $[c,d]_{\itLamb}$ of $\itLamb$,
such that there exists an embedding $h:[c,d]_{\itLamb}\to \kk$ satisfying the following conditions:
\begin{enumerate}[label=\text{\rm(X\arabic*)}]
  \item $h$ has an order-preserving extension $\tilde{h}: X \to\kk$, $X\subseteq\itLamb$, which is a bijection;
    \label{X1}
  \item $X$ is a ring whose multiplication is defined as $x_1 x_2 := h^{-1}(h(x_1)h(x_2))$
    and whose addition is defined as $x_1+x_2 := h^{-1}(h(x_1)+h(x_2))$;  \label{X2}
  \item $X$ is a $\kk$-linear space whose scalar multiplication is defined as
    $kx := h^{-1}(kh(x))$. \label{X3}
\end{enumerate}
Then $X$ is an algebra and $\tilde{h}$ is a homomorphism of algebras,
and we obtain an integral Banach module category $(\scrA^1, \tilde{h}, [c,d]_{\itLamb}, \calH) =: \widetilde{\scrA^1}$,
where $\calH$ is either a Hausdorff measure or a Lebesgue measure defined on $X$.
Let $\frakI$ be a partial ordered subset of $\widetilde{\frakI} := \{ [t_1, t_2]_{\itLamb} \subseteq [r,s]_{\itLamb} \mid
c \preceq t_1 \preceq t_2 \preceq d \}$ with the partial order ``$\subseteq$''.
An \defines{integral partial ordered set} ($=$ i.poset for short) of $f$, say $\w{T}_{\frakI}(f)$, is a poset whose elements are pairs of the form
$([t_1, t_2]_{\itLamb}, \w{T}_{t_1}^{t_2}(f))$ where:
\begin{enumerate}[label=\text{\rm(\arabic*)}]
  \item $\w{T}_{t_1}^{t_2}:= \w{T}_{(\kk,\mu([t_1,t_2]_{\itLamb}),m)}$
    is the homomorphism in $\widetilde{\scrA^1}$ from the initial object in $\widetilde{\scrA^1}$ to the object $(\kk,\mu([t_1,t_2]_{\itLamb}),m)$;
  \item and $[t_1,t_2]_{\itLamb}$ runs through all the elements of $\frakI$.
\end{enumerate}

\begin{remark} \rm
The measure $\calH([c,d]_{\itLamb})$ of $[c,d]_{\itLamb}$
can be regarded as a generalization of the length of $[c,d]_{\itLamb}$,
and $\calH$ and $\mu$ can usually be independent of each other.
In the case where $\itLamb$ possesses some good properties,
$\calH$ induces a Lebesgue measure $\mu$ defined on $\itLamb$.
\end{remark}

In general, $\calH$ and $\mu$ may be unrelated.
Therefore, in the generalized L-condition, i.e., in the cases \checks{when} the following conditions hold:
\begin{enumerate}[label=\text{\rm(GL\arabic*)}]
  \item $p=1$, \label{L1}
  \item $\kk=\RR$ or $\CC$, \label{L2}
  \item $\homo: \itLamb \to \kk$ is an epimorphism of $\kk$-algebras, \label{L3}
  \item $\II_{\itLamb} = [r,s]_{\kk}^{\times 2^{\dim\itLamb}}$ ($r \preceq s$), \label{L4}
  \item $\xi \in (r,s)_{\kk} := [r,s]_{\kk}\backslash\{r,s\}$, \label{L5}
  \item $\mu$ is a Lebesgue measure, \label{L6}
\end{enumerate}
any integrals with variable upper limit $\displaystyle (\widetilde{\scrA^1})\int_c^x f\dd\calH$ ($r\preceq c\preceq x \preceq d\preceq s$) of $f\in \wS$ is the i.poset
\begin{align}\label{def:var-int}
  \w{T}_c^x(f):=\w{T}_{\frakI=\{ [c, x]_{\itLamb} \mid
c \preceq x\preceq d \}}(f)
= \left\{\left.\left( [c,x]_{\itLamb}, ~ (\widetilde{\scrA^1})\int_c^{x} f \dd\calH\right)
~\right|~ c \preceq x\preceq d
\right\}.
\end{align}
In particular, we have the following two special cases:
\begin{enumerate}[label=\textrm{\checks{(\arabic*)}}]
\item
If \ref{L2} is replaced by $\itLamb=\kk=\RR$,
\ref{L3} is replaced by $\homo=\ident_{\itLamb}=\ident_{\kk}$,
and \ref{L4} is replaced by $\xi = \frac{r+s}{2}$,
then we \checks{say} that $\scrA^1$ satisfies L-condition.
In this case, we have $\mu=\calH$ and
\begin{center}
  $\displaystyle (\widetilde{\scrA^1})\int_c^x f\dd\calH = (\Lebesgue)\int_c^x f\dd\mu$
\end{center}
is a Lebesgue integration.
\item
If \ref{L1}--\ref{L6} hold, $\itLamb=\CC=\kk$, and $[c,d]_{\itLamb}$ is a curve
whose starting point and ending point respectively are $c, d\in \CC$,
then we have
\begin{center}
  $\w{T}_c^x(f) = \displaystyle (\widetilde{\scrA^1})\int_c^x f\dd\calH
   = \int_c^x (u(x,y)+v(x,y)\mathrm{i})\dd s$,
\end{center}
which is a line integration. Here, $u(x,y)$ and $v(x,y)$ are function of $\RR \times \RR \to \RR$,
and $\calH$ is a Hausdorff measure that dose not coincide with $\mu$.
\end{enumerate}

Leinster provided another description of integrals with variable upper limit by using $\scrA^1$ in the L-condition.
He consider the triple $(C_*([c,d]), \id:x\mapsto x, \eta)$, and point out that it is an object in $\scrA^1$
(due to Mark Meckes, see \cite[Section 2]{Lei2023FA}), where
$C_*([c,d])$ is the Banach module of continuous functions $F:[c,d]\to \RR$ such that $F(c)=0$ holds
(the norm defined on it is the sup norm) and $\eta$ is defined as
\[\eta: C_*([c,d])^{\oplus 2} \to C_*([c,d]), \]
\[ \eta(f,g)(x) := \begin{cases}
  \displaystyle \frac{1}{2}f(2x-c), &
  \displaystyle c\=< x \=< \frac{c+d}{2}; \\
  \displaystyle \frac{1}{2}(f(d)+g(2x-d)), &
  \displaystyle \frac{c+d}{2}\=< x \=< d,
\end{cases} \]
and showed the following result.

\begin{theorem}[{Leinster--Meckes \cite[Proposition 2.4]{Lei2023FA}}]\label{thm:LM2023}
Under the L-condition, there is a unique homomorphism
\[\w{T}_{(C_*([c,d]), \id, \eta)} \in \Hom_{\scrA^1}(
(L^1([c,d]), d-c, \gamma_{\frac{c+d}{2}}),
(C_*([c,d]), \id, \eta))\]
sending each $f\in C_*([c,d])$ to the Lebesgue integral with variable upper limit
$\displaystyle (\Lebesgue)\int_c^x f\dd\mu$.
Here, $L^1([c,d]) \cong \wS$ is a initial object in $\scrA^1$.
\end{theorem}

\checks{One of the results in} \cite{LLL2025cat3} \checks{was} to redescribe the above result by using i.poset,
see \cite[Corollary 4.3]{LLL2025cat3}.


\begin{remark}\label{rmk:FTC} \rm
The core content of Theorem \ref{thm:LM2023} is the following:
under the L-condition, the integral with variable upper limit function
$F(x) := \displaystyle (\scrA^1)\int_c^x f \dd\calH$ is absolutely continuous,
and, in analysis, its (weak) derivative recovers $f$ \checks{(a.e.)}.
That is, $\Dop$ and $\w{T}_c^x$ are inverses of each other in an appropriate sense.
This inverse relationship will be extended to $V$-valued functions in Section \ref{sec:equiv}.
\end{remark}

\section{$V$-valued functions and categorical ODEs}\label{sec:V-ODE}
This section is one of the core innovations of the paper:
we introduce $V$-valued function spaces and give the rigorous definition of categorical ODEs.

\subsection{$V$-valued elementary simple functions}\label{subsec:V-simple}

\begin{definition}\label{def:V-simple} \rm
Let $V$ be a Banach $\itLamb$-module. A \defines{$V$-valued elementary simple function} defined on $\II_\itLamb$ is a finite sum
\begin{align}\label{eq:V-simple}
  f = \sum_{i=1}^{N} v_i\id_{I_i},
\end{align}
where $v_i \in V$, $I_i \subseteq \II_{\itLamb}$, and all subsets $I_i$ of $\II_{\itLamb}$ are pairwise disjoint
(as in Definition \ref{def:simple-func}).
The set of all $V$-valued elementary simple functions is denoted by $\SV$.
\end{definition}

\begin{proposition}\label{prop:V-module}
$\SV$ has a left $\itLamb$-action given by
\begin{align}\label{eq:V-module-action}
  (a\checks{\cdot_V} f)(x) := a \cdot_V f(x), \quad \forall a \in \itLamb, f \in \SV, x \in \II_\itLamb,
\end{align}
such that it is a $\itLamb$-module,
where $a \cdot_V f(x)$ is obtained by the left $\itLamb$-action defined on $V$ naturally,
since for each $x\in\II_{\itLamb}$, $f(x)$ is an element in $V$,
i.e., $a\checks{\cdot_V} f = (1_{\itLamb}\otimes a)f = (1_{\itLamb}\otimes a)(\sum_i f_i\otimes e_i)
= \sum_i f_i\otimes ae_i$ holds {\rm(}\checks{where} $\{e_1,\ldots,e_m\}$ is a basis of $V${\rm)}.
Furthermore, $\SV$ is a $\homo$-normed $\itLamb$-module with the norm
\begin{align}\label{eq:V-norm}
  \|f\|_{\SV} := \sum_{i=1}^{N} \|v_i\|_V \cdot \mu(I_i).
\end{align}
\end{proposition}

\begin{proof}
We first verify the $\itLamb$-module structure of $\SV$. Let $a, a' \in \itLamb$, $f, g \in \SV$, and $k \in \kk$,
we have the following three facts:
\begin{enumerate}[label=\text{\rm(\arabic*)}]
  \item $a \cdot_V f(x) \in V$ holds for all $x$, so $a \cdot_V f$ is still a $V$-valued function.
  \item $\big(a \cdot_V (a' \cdot_V f)\big)(x) = (aa') \cdot_V f(x) = ((aa') \cdot_V f)(x)$ (associativity).
  \item $(a \cdot_V (f + g))(x) = a \cdot_V (f(x) + g(x)) = a \cdot_V f(x) + a \cdot_V g(x)$ (distributivity).
\end{enumerate}
Hence $\SV$ is a left $\itLamb$-module.

Next, let $f = \displaystyle \sum_i v_i \cdot \id_{I_i}$.
Then $a \cdot_V f = \displaystyle \sum_i (a \cdot_V v_i) \cdot \id_{I_i}$. Thus,
\begin{align*}
 \|a \cdot_V f\|_{\SV} & = \sum_i \|a \cdot_V v_i\|_V \mu(I_i) \\
   & \mathop{\=<}\limits^{\star} \sum_i |\homo(a)|~\|v_i\|_V \mu(I_i) \\
   & = |\homo(a)| ~\|f\|_{\SV},
\end{align*}
where $\star$ \checks{holds by} Definition \ref{def:normed-module}.
The remaining properties of the norm\checks{: positive definiteness, homogeneity, and the triangle inequality,} can be verified directly.
\end{proof}

\begin{notation}[$V$-valued completion]\label{def:V-completion} \rm
The completion of $\SV$ is denoted by
\begin{center}
  $\wSV $:= the completion of $\SV$ with respect to $\|\cdot\|_{\SV}$,
\end{center}
i.e., the space of equivalence classes of Cauchy sequences in $\SV$.
Then $\wSV$ is a Banach $\itLamb$-module.
\end{notation}

When $V$ is a finite-dimensional $\kk$-linear space (with $\kk$-basis $\{e_1, \ldots, e_m\}$),
$\SV$ is naturally isomorphic to $\bfS_\homo(\II_\itLamb) \otimes_\kk V$,
and after completion $\wSV \cong \wS \otimes_\kk V$ (completed tensor product) if $\kk$ is complete.
In this fact, a $V$-valued function $f$ can be written as
\begin{align}\label{eq:tensor-decomp}
  f = \sum_{j=1}^{m} f_j \otimes e_j, \quad f_j \in \wS.
\end{align}

\begin{example}\rm
A well-known example is the function with complex variable:
each function with complex variable $f: \mathbb{C} \to \mathbb{C}$, $z\mapsto f(z)$ can be written as
\[ f(x,y) = u(x,y) + v(x,y) \mathrm{i} =  u(x,y) \otimes 1 + v(x,y)\otimes \mathrm{i}, \]
where $u(x,y)$ and $v(x,y)$ are function of the form $\RR^2 \to \RR$.
In this example, $\itLamb=\CC =\RR\times \RR\mathrm{i}$ is an $\RR$-algebra,
$V=\CC\cong \RR^2$ is an $\RR$-module,
$\bfS_{\mathrm{id}}(\II_{\itLamb})$ is the set of all
elementary simple functions defined on $\II_{\itLamb} = [c,d]^{\times 2}$,
and we have
\[
\bfS_{\mathrm{id}}(\II_{\itLamb})\otimes_{\RR}\CC
= (\bfS_{\mathrm{id}}(\II_{\itLamb}) \otimes_{\RR} \RR)
\oplus (\bfS_{\mathrm{id}}(\II_{\itLamb}) \otimes_{\RR} \RR\mathrm{i})
\mathop{\cong}\limits^{\sigma} \bfS_{\mathrm{id}}(\II_{\itLamb}, \CC^2)
\]
Here, each element in $\bfS_{\mathrm{id}}(\II_{\itLamb})$ is a map of the form $\RR^2\to \RR$,
and each element in $\bfS_{\mathrm{id}}(\II_{\itLamb},\CC)$ is a map of the form $\RR^2\to \CC$.
The isomorphism $\sigma$ provides a correspondence
\[ u(x,y)+v(x,y)\mathrm{i} \mapsto f(z) \]
from $\w{\bfS_{\mathrm{id}}(\II_{\itLamb})}\otimes_{\RR}\CC$ to $\w{\bfS_{\mathrm{id}}(\II_{\itLamb},\CC)}$,
where $\Re(f(z))=u(x,y)$ and $\Im(f(z))=v(x,y)$.
\end{example}

\subsection{$V$-valued function spaces as objects of $\scrA^p$}\label{subsec:V-object}

In order to treat $V$-valued functions within \checks{our categorical} framework, we need to show that $V$-valued function spaces \checks{are objects} of $\scrA^p$.

Let $\gamma_\xi: \bfS_\homo(\II_\itLamb)^{\oplus_p 2^n} \to \bfS_\homo(\II_\itLamb)$ be the juxtaposition map given in the $\scrA^p$-initial object in $\scrN^p$.
\checks{Notice that since} each linear space $V$, as a $\kk$-module, is flat,
then for any homomorphism $h: X\to Y$ of $\itLamb$-modules $X$ and $Y$,
$h\otimes_{\kk} V$ has a canonical composition
\[ \xymatrix{
 X\otimes_{\kk} V \ar[rr]^{h\otimes_{\kk} V} \ar[rd]_{h_1} & & Y\otimes_{\kk} V \\
 & \im(h\otimes_{\kk}V) \ar[ru]_{h_2} &} \]
such that
$\im(h\otimes_{\kk}V) \cong \im(h)\otimes_{\kk}V$,
and $h_1$ and $h_2$ are introduced by the tensor functor $-\otimes_{\kk}V$ applying to $\hbar:X\to \im(h)$
and $\emb: \im(h) \mathop{\to}\limits^{\subseteq} Y$.
Here, $\emb\compos \hbar$ is the canonical
 composition of $h$.

\begin{definition}\rm
We define the \defines{$V$-valued juxtaposition map} as the map $\gamma_{\xi,V}:= \gamma_{\xi}\otimes_{\kk}V$.
That is, $\gamma_{\xi,V}$ is \checks{the} homomorphism of $\itLamb$-modules \checks{defined} as follows:
\begin{align*}
  \gamma_{\xi,V}: \SV^{\oplus_p 2^n} \to \SV, \quad
  (f_1, \ldots, f_{2^n}) \mapsto \gamma_{\xi,V}(f_1, \ldots, f_{2^n}),
\end{align*}
where the right-hand side applies $\gamma_\xi$ componentwise to the $\kk$-valued components.
More precisely, in the tensor product representation (\ref{eq:tensor-decomp}),
writing $f_l = \displaystyle \sum\nolimits_j f_{l,j} \otimes e_j$, we have
\begin{align*}
    (f_1, \ldots, f_{2^n})
& = \bigoplus_{l=1}^{2^n} f_l := \sum_{l=1}^{2^n} (0,\ldots,0, f_l,0,\ldots,0) \\
& = \bigoplus_{l=1}^{2^n} \sum_{j=1}^m f_{l,j} \otimes e_j \\
& = \sum_{j=1}^m \bigoplus_{l=1}^{2^n} f_{l,j} \otimes e_j \\
& = \sum_{j=1}^m \bigg(\bigoplus_{l=1}^{2^n} f_{l,j}\bigg) \otimes e_j \\
& = \sum_{j=1}^m (f_{1,j},\ldots,f_{2^n,j}) \otimes e_j,
\end{align*}
and then
\begin{align} \label{eq:V-valued}
  \gamma_{\xi,V}(f_1, \ldots, f_{2^n}) := \sum_{j=1}^m \gamma_\xi(f_{1,j}, \ldots, f_{2^n,j}) \otimes e_j.
\end{align}
\end{definition}

\begin{proposition}\label{prop:V-object}
Let $V$ be a $\homo$-normed $\itLamb$-module containing an element $v_0 \in V$ with $\|v_0\|_V \=< 1$.
Define $\id_{\II_{\itLamb},V} := \id_{\II_\itLamb} \otimes v_0$
{\rm(}i.e., $x \mapsto v_0$ for all $x \in \II_\itLamb${\rm)}.
Then $\id_{\II_{\itLamb},V} \in \SV$, and the triple $(\wSV,\id_{\II_\itLamb,V},\w{\gamma}_{\xi,V})$ is an object in $\scrA^p$.
\end{proposition}

\begin{proof}
It is clear that $\id_{\II_{\itLamb},V} \in \SV$. Now we prove that
$(\wSV,\id_{\II_\itLamb,V},\w{\gamma}_{\xi,V})$ is an object in $\scrA^p$.
We need to verify the conditions in Definition \ref{def:Ap}.

First, since $\wS$ and $V$ are $\itLamb$-modules, then so is $\wS\otimes_{\kk}V$.
Here, $\wS\otimes_{\kk}V$ has two left $\itLamb$-actions
\begin{center}
  $a(f\otimes v) = af\otimes v$ and
  $a\cdot_V(f\otimes v) := (1_{\itLamb}\otimes a)(f\otimes v), $
\end{center}
or equivalently, $\wS$ is a $(\itLamb, \itLamb^{\mathrm{op}})$-bimodule
whose left $\itLamb$-action and right $\itLamb^{\mathrm{op}}$-action are given by the above first and second equations, respectively.
Here, the second action is given by Proposition \ref{prop:V-module}.
In the sense of the second, fixing a basis of $V$ that is $\{e_i,\ldots,e_m\}$, we have
$v = \displaystyle \bigoplus_{i=1}^m k_ie_i:=\sum_{i=1}^m k_ie_i$
(\checks{where} $k_1,\ldots, k_m\in \kk$, this sum is a formal sum ``$\bigoplus$''), and then
\[ a\cdot_V(f\otimes v) = (1_{\itLamb}\otimes a)\bigg(\sum_{i=1}^m f_i\otimes e_i\bigg)\otimes \bigoplus_{i=1}^m k_ie_i. \]
Notice that $k_ie_i = k_i(1\otimes e_i) = k_i\otimes e_i$ holds for all $k_i\in\kk$
since the $\kk$-linear space $V$ can be viewed as the $\kk$-tensor
\[ V = \bigoplus_{i=1}^m \kk e_i \cong \kk \otimes_{\kk} V
= \kk \otimes_{\kk} \bigg(\bigoplus_{i=1}^m \kk e_i\bigg)
= \bigoplus_{i=1}^m \kk \otimes_{\kk} \kk e_i
= \bigoplus_{i=1}^m \kk (1\otimes_{\kk}e_i) \]
where for any $i$ the multiplication of $k_i\otimes e_i$ and $k_i'\otimes e_i$ (for any $k_i,k_i'\in\kk$)
is $k_ik_i'\otimes e_i$ induced by the scalar multiplication of $\kk$-linear space, we obtain
\[ a\cdot_V(f\otimes v) = \sum_{i=1}^m k_if_i\otimes ae_i
= \sum_{i=1}^m \bigg(k_if_i\otimes \bigoplus_{t=1}^m k_{it}e_t\bigg)
= \bigoplus_{t=1}^m  \bigg(\sum_{i=1}^m k_{it}k_if_i\bigg)\otimes e_t. \]
Here, $ae_i = \displaystyle\bigoplus_{t=1}^m k_{it}e_t$ for some $k_{it} \in \kk$ ($1\=<i,t\=< m$).
The above equation means that $a\cdot(f\otimes v)$ (or more generally, $f\otimes v$)
can be expressed as a sum of tensors in different ways.
Then, for the well-definedness of the norm $\|\cdot\|_{\wSV}$, we define
\[ \|\cdot\|_{\wSV}: \wSV \cong \wS\otimes_{\kk}V \to \RR^{\>=0}, \]
\[f\otimes v \mapsto
  \inf\left\{
    \left.
    \sum_t\|f_t\|_{\wS}\|v_t\|_V
    \right|
    f\otimes v = \sum_t f_t\otimes v_t
  \right\},\]
then $\|\cdot\|_{\wSV}$ is a norm defined on the tensor $\wSV$,
this norm is said to be a \defines{projected tensor product norm} in analysis.
On the other hand, for any $f\otimes v = \displaystyle \sum_t f_t\otimes v_t$, we have
\[ a\cdot_V (f\otimes v) = \sum_t f_t\otimes av_t. \]
Since \[\displaystyle \sum_t \|f_t\|_{\wS} \| av_t\|_V
\mathop{\=<}\limits^{\star} \sum_t \|f_t\|_{\wS} \homo(a) \|v_t\|_V
 = \homo(a) \sum_t \|f_t\|_{\wS}\|v_t\|_V, \]
we obtain
\[ \|a\cdot_V (f\otimes v)\|_{\wSV} \=< \homo(a) \sum_t \|f_t\|_{\wS}\|v_t\|_V \]
for all sums $\displaystyle \sum_t \|f_t\|_{\wS}\|v_t\|_V$ with $f\otimes v= \displaystyle \sum_t f_t\otimes v_t$,
where $\star$ holds by using Definition \ref{def:normed-module}.
It follows that
\[ \|a\cdot_V (f\otimes v)\|_{\wSV} \=< \homo(a) \|f\otimes v\|_{\wSV}, \]
i.e., $\wS\otimes_{\kk}V$ is a $\homo$-normed module.
Notice that $\kk$ is complete and $V$ is finite-dimensional,
then $\wS\otimes_{\kk}V$ is complete, and so, $\wSV \cong \wS\otimes_{\kk}V$ is complete.
Thus, $\wSV \cong \wS\otimes_{\kk}V$ is a Banach module.

Second, we have $\id_{\II_\itLamb,V} \in \wSV$, and
$\|\id_{\II_\itLamb,V}\|_{\wSV} = \mu(\II_\itLamb)\cdot \|v_0\|_V \=< \mu(\II_\itLamb)$.

Third, since $(\wS, \id, \w{\gamma}_{\xi})$ is an object in $\scrA^p$ (see Theorem \ref{thm:initial},
or see Remark \ref{rmk:initial} \ref{thm:initial adm 1}),
we have $\gamma_\xi(\id_{\II_\itLamb}, \ldots, \id_{\II_\itLamb}) = \id_{\II_\itLamb}$.
Thus, we obtain the following equation
\begin{align*}
    \w{\gamma}_{\xi,V}(\id_{\II_\itLamb,V}, \ldots, \id_{\II_\itLamb,V})
& = \gamma_{\xi,V}(\id_{\II_\itLamb} \otimes v_0, \ldots, \id_{\II_\itLamb} \otimes v_0) \\
& = \gamma_{\xi}(\id_{\II_\itLamb}, \ldots, \id_{\II_\itLamb})  \otimes v_0 \\
& = \id_{\II_\itLamb} \otimes v_0 \\
& = \id_{\II_\itLamb,V}
\end{align*}
by using (\ref{eq:V-valued}).

Finally, the commutativity of $\w{\gamma}_{\xi, V}$ with inverse limits of Cauchy sequences is inherited from that of $\w{\gamma}_{\xi}$ with inverse limits of Cauchy sequences.
Therefore, $(\wSV,\id_{\II_\itLamb,V},$ $\w{\gamma}_{\xi,V})$ is an object in $\scrA^p$.
\end{proof}

\subsection{$V$-valued differential operator and integral}\label{subsec:DV-TV}

Let $V$ be a Banach $\itLamb$-module with $\kk$-basis $\{e_1, \ldots, e_m\}$.
Every function $f \in \wSV$ can be written as
\begin{align}\label{eq:function}
  f = \displaystyle \sum_{j=1}^{m} f_j \otimes e_j \in \wSV
\end{align}
since $\wSV \cong \wS \otimes_{\kk} V$, where each $f_j$ is a function lying in $\wS$.
The conclusion stated in Theorem~\ref{thm:D} can be generalized to the case where $\itLamb$ is an algebra.
In this case, the totally ordered subset $S:=[c,d]_{\itLamb}\subseteq \itLamb$
can be taken as a one-dimensional smooth manifold in $\itLamb$,
i.e., $[c,d]_{\itLamb}$ is a compact differentiable manifold given by a map $S:[0,1]\to [c,d]_{\itLamb}$.
Then we can define the differential $D$ for absolutely continuous functions $f$ defined on $[c,d]_{\itLamb}$.
Next, we define $V$-valued homomorphic operator.

\begin{definition}\label{def:DV} \rm
\checks{The \defines{$V$-valued homomorphic operator} $\Dop_V$ is defined on absolutely continuous functions $f$ given in (\ref{eq:function}), and is defined as follows:}
\begin{align}\label{eq:DV}
  \Dop_V f := \sum_{j=1}^{m} (h(f_j)) \otimes e_j,
\end{align}
where $h$ is a homomorphism lying in (\ref{eq:D}).
In particular, if $\scrA^p$ satisfies generalized L-condition and
$h = \Dop$ is the differential operator from Theorem \ref{thm:D}
(we suppose $[c,d]_{\itLamb}$ is a compact differentiable manifold in this paper),
then (\ref{eq:DV}) is said to be a \defines{$V$-valued differential operator}.
\end{definition}

Next, we always assume that the operator $\Dop_V$ is a $V$-valued differential operator.
Clearly, we have
\[ \sum_{j=1}^{m} (\Dop f_j) \otimes e_j \in \wS\otimes_{\kk} V
\mathop{\cong}\limits^{\text{Prop. \ref{prop:V-object}}} \wSV \]
since all functions $f_j$ lie in $\wS$.
The following proposition shows that the definition of \checks{the} $V$-valued differential operator
given in Definition \ref{def:DV} is well-defined, and $\DV$ is a homomorphism of $\itLamb$-modules lying in $\End_{\itLamb}(\wSV)$.

\begin{proposition}\label{prop:DV-welldefined}
\hfill
\begin{enumerate}[label=\textrm{\rm(\arabic*)}]
  \item \label{prop:DV-welldefined 1}
    The definition of $\DV$ is independent of the choice of $\kk$-basis of $V$;
  \item \label{prop:DV-welldefined 2}
    $\DV$ is $\kk$-linear, i.e., $\forall k_1, k_2 \in \kk$, we have $\DV(k_1 f + k_2 g) = k_1 \DV f + k_2 \DV g$;
  \item \label{prop:DV-welldefined 3}
    \checks{$\DV$} is a homomorphism of $\itLamb$-modules.
\end{enumerate}
\end{proposition}

\begin{proof}
(1) Let $\mathfrak{B}:=\{e_1, \ldots, e_m\}$ and $\mathfrak{B}':=\{e_1', \ldots, e_m'\}$ be two $\kk$-\checks{bases} of $V$,
then we have \[e_j = \sum_l p_{jl} e_l'\] for any $e_j\in \mathfrak{B}$\checks{, and some $p_{jl}\in\kk$}.
Thus, we have
\[ f = \sum_j f_j \otimes e_j = \sum_l \Big(\sum_j p_{jl} f_j\Big) \otimes e_l'. \]
In the basis $\mathfrak{B}'$,
\begin{align*}
 \DV f = \sum_l \Dop\Big(\sum_j p_{jl} f_j\Big) \otimes e_l'
 = \sum_l \Big(\sum_j p_{jl} \Dop f_j\Big) \otimes e_l'
 = \sum_j (\Dop f_j) \otimes e_j,
\end{align*}
which coincides with the definition in the basis $\mathfrak{B}$.

(2) This follows directly from the $\kk$-linearity of $\Dop$ (Theorem \ref{thm:D} \ref{thm:D 2}) and the bilinearity of the tensor product.

(3) We need prove that $\DV(a \checks{\cdot_V} f) = a \checks{\cdot_V} \DV f$ holds for all $a \in \itLamb$.
Take $a \in \itLamb$ and $f = \displaystyle \sum_j f_j \otimes e_j \in \wSV$, we have
\[a \checks{\cdot_V} f = \sum_j f_j \otimes ae_j \]
by the fact given in Proposition \ref{prop:V-object} that $\wSV$ is a left $\itLamb$-module,
and then we obtain
\[ a \checks{\cdot_V} f = \sum_l \Big(\sum_j f_j\Big) \otimes ae_l.\]
Therefore,
\begin{align*}
  \DV(a \cdot_V f)
  &= \sum_l \Dop\Big(\sum_j f_j\Big) \otimes ae_l \\
  &= \sum_l \sum_j \Dop(f_j) \otimes ae_l \\
  &= a\cdot_V \Big(\sum_l\sum_j (\Dop f_j) \otimes e_l\Big) \\
  &= a\cdot_V \DV f.
\end{align*}
\end{proof}

\begin{definition}\label{def:TV} \rm
For a function $f = \displaystyle \sum_{j=1}^{m} f_j \otimes e_j \in \wSV$,
the \defines{$V$-valued integral with variable upper limit} of $f$ is defined as
\begin{align}\label{eq:TV}
  \w{T}_{x_0,V}^{x}(f) := \sum_{j=1}^{m} \w{T}_{x_0}^{x}(f_j) \otimes e_j
  = (\w{T}_{x_0}^{x}(f_1), \ldots, \w{T}_{x_0}^{x}(f_m))
  =: \bigoplus_{j=1}^m \w{T}_{x_0}^{x}(f_j) e_j  \in V,
\end{align}
where $[x_0,x]_{\itLamb}$ is a continuous totally ordered subset of $\itLamb$,
and $\w{T}_{x_0}^x$ is defined from (\ref{def:var-int}).
\end{definition}

Obviously, it is trivial that $\w{T}_{x_0,V}^{x_0}(f) = 0$ holds for all $f \in \wSV$.

\begin{proposition}\label{prop:TV-properties}
\hfill
\begin{enumerate}[label=\text{\rm(\arabic*)}]
  \item $\|\w{T}_{x_0,V}^{x}(f)\|_V \=< \displaystyle \sum_j \w{T}_{x_0}^{\alpha}(\|f_j\|_V)$
   ~{\rm(}where, for each $\varphi\in\wSV$, $\|\varphi\|_V$ is the function defined by $x\mapsto \|\varphi(x)\|_V$ for all $x${\rm)}; \label{prop:TV-properties 1}
  \item $\w{T}_{x_0,V}^{x}:\wSV \to V$ is a $\kk$-linear map; \label{prop:TV-properties 2}
  \item furthermore, $\w{T}_{x_0,V}^{x}$ is a homomorphism of $\itLamb$-modules. \label{prop:TV-properties 3}
\end{enumerate}
\end{proposition}

\begin{proof}
(1) By the triangle inequality
\begin{align*}
  \| \w{T}_{x_0,V}^{x}(f) \|_V
 = \left\| \bigoplus_j \w{T}_{x_0}^{x}(f_j) e_j \right\|_V
 \=< \sum_j \|\w{T}_{x_0}^{x}(f_j)\|_V
 \=< \sum_j \w{T}_{x_0}^{x}(\|f_j\|_V),
\end{align*}
we have
\[ \| \w{T}_{x_0,V}^{\alpha}(f) \|_V \=< \sum_j \w{T}_{x_0}^{\alpha}(\|f_j\|_V).\]
\checks{Holds} for all $\alpha\in [x_0,x]_{\itLamb}$.

(2) holds since $\w{T}_{x_0}^x$ is $\kk$-linear (this proof is written in \cite[Corollary 4.3]{LLL2025cat3}).

(3) The proof is similar to that of Proposition \ref{prop:DV-welldefined} \ref{prop:DV-welldefined 3}.
\end{proof}

\subsection{Ordinary differential equations}\label{subsec:ODE-def}

We now state the central definition of this paper.

\begin{definition}\label{def:ODE} \rm
Let $V$ be a Banach $\itLamb$-module (\checks{with basis} $\{e_1, \ldots, e_m\}$),
$G: \itLamb \times V \to V$ be a continuous map, $f_0$ be an element in $V$,
and $x_0$ lie in a given continuous totally ordered subset $[c,d]_{\itLamb}$ of $\II_{\itLamb}$.
The \defines{{\rm(}$V$-valued{\rm)} initial value problem} is defined as
\begin{align}\label{eq:ODE-general}
  \DV f(x) = G(x, f(x)), \quad f(x_0) = f_0,
\end{align}
where $f \in \wSV$. Furthermore, if there is a map $A: [c,d]_{\itLamb} \to \itLamb$ such that
$G(x, f(x)) = -A(x)\cdot_V f(x) + g(x)$ holds for all $x\in[c,d]_{\itLamb}$, then (\ref{eq:ODE-general}) becomes
\begin{align}\label{eq:ODE-linear}
  \DV f(x) +  A(x) \cdot_V f(x) = g(x), \quad f(x_0) = f_0,
\end{align}
which is said to be a \defines{linear ordinary differential equations} (=linear ODE for short).
Here,
\begin{enumerate}[label=\textrm{\checks{(\arabic*)}}]
  \item $A(x)$ is called the \defines{coefficient function} of (\ref{eq:ODE-linear});
  \item $g \in \wSV$ is called the \defines{non-homogeneous term} of (\ref{eq:ODE-linear});
  \item for each $t\in \II_{\itLamb}$, $A(t) \cdot_V f(t)=(1_{\itLamb}\otimes A(t))f(t)$ is induced by
    the left $\itLamb$-action defined on $V$ ($1_{\itLamb}$ is the identity in $\itLamb$),
  \item and, when write $f$ as $f = \displaystyle\sum_{i=1}^m f_i\otimes e_i$, we have
    \[(1_{\itLamb}\otimes A(x)) \checks{f}
     = \sum_{i=1}^m (1_{\itLamb}\otimes A(x))(f_i\otimes e_i)
     = \sum_{i=1}^m f_i\otimes A(x)e_i\]
     by using the multiplication of tensor.
\end{enumerate}
For simplicity, we write $(1_{\itLamb}\otimes A(x)) f$ as $A(x)\cdot_V f$ in our paper.
When $g = 0$, equation (\ref{eq:ODE-linear}) is called a
\defines{homogeneous linear ordinary differential equations}.
\end{definition}

The framework of Definition \ref{def:ODE} naturally encompasses higher-order ODEs,
i.e., an \defines{$n$-th order linear ODE}
\begin{align}\label{eq:ODE-higher}
  \DV^n f + A_{n-1}(x) \cdot_V \DV^{n-1} f + \cdots + A_1(x) \cdot_V \DV f + A_0(x) \cdot_V f = g(x)
\end{align}
can be reduced to a first-order system (\ref{eq:ODE-linear}) via the standard order reduction:
set $\pmb{f} = (f, \DV f, \ldots, \DV^{n-1} f)^{\top} \in \wSV{}^{\oplus n}$;
then (\ref{eq:ODE-higher}) is equivalent to the $V^{\oplus n}$-valued first-order ODE
\begin{align}\label{eq:ODE-reduction}
  \Dop_{V^{\oplus n}} \pmb{f}(x) + \pmb{A}(x)\cdot_V\pmb{f}(x) = \pmb{g}(x),
\end{align}
where $\pmb{A}(x) = (A_1(x),\ldots, A_n(x))$ is the corresponding companion matrix.
Therefore, subsequent discussions can focus on the first-order system (\ref{eq:ODE-linear}) without loss of generality
(note that each $x\in[c,d]_{\itLamb}$ is a vector since $\itLamb$ is a $\kk$-linear space).

\section{Equivalence of differential and integral equations}\label{sec:equiv}

In this section we prove the equivalence between the $V$-valued ODE (as a differential equation form) and the corresponding integral equation form.
All theorems used in this section originate from \cite{LLHZ2025, GLLWpre, LLL2025cat3} or their consequences.

\subsection{The inverse relationship between $\Dop$ and $\w{T}_{x_0}^x$}\label{subsec:inverse}

Now we provide the first theorem of our paper, which is the key technical tool for the Equivalence Theorem (see Theorem \ref{thm:equivalence}).

\begin{lemma} \label{lemm:V-FTC 1}
Assume that $\scrA^1$ satisfies generalized L-condition {\rm(}see \ref{L1}--\ref{L6}{\rm)}.
For each $f \in \wSV$, if we define $F(x):=\w{T}_{x_0,V}^{x}(f)$, then we have
\begin{align*}
  \DV F = f(x)  \hspace{0.5cm} \text{{\rm(}a.e.{\rm)}}
\end{align*}
holds for all $x_0, x \in [c,d]_{\itLamb}$.
\end{lemma}

\begin{proof}
Assume $f = \displaystyle \sum_{j=1}^m f_j \otimes e_j \in \wSV$. Then we have
\begin{align*}
  \DV\big(\w{T}_{x_0,V}^x(f)\big)
 = \DV\Big(\sum_j \w{T}_{x_0}^x(f_j) e_j\Big)
 = \sum_j \Dop\big(\w{T}_{x_0}^x(f_j)\big) \otimes e_j
\end{align*}
by Definitions \ref{def:DV} and \ref{def:TV}.
For each $f_j \in \wS$, we have that
\begin{center}
  $x \mapsto \w{T}_{x_0}^x(f_j)$
\end{center}
is absolutely continuous.
Then its weak derivative recovers $f_j$, i.e., we have
\begin{center}
  $\Dop(\w{T}_{x_0}^x(f_j)) = f_j(x)$ \checks{(a.e.)}.
\end{center}
It follows $\displaystyle \DV\big(\w{T}_{x_0,V}^x(f)\big) = \sum_j f_j(x) \otimes e_j = f(x)$ as required.
\end{proof}

\begin{lemma} \label{lemm:V-FTC 2}
\checks{Keeping} the notations from Lemma \ref{lemm:V-FTC 1}.
For any absolutely continuous $f = \displaystyle \sum_j f_j \otimes e_j: \II_{\itLamb} \to V$
in $\wSV$ {\rm(}i.e., all components $f_j$ are absolutely continuous{\rm)}, we have
\begin{align*}
\w{T}_{x_0,V}^{x}(\DV f) = \sum_j (f_j(x)-f_j(x_0))e_j
\end{align*}
holds for all $x_0, x\in [c,d]_{\itLamb}$.
\end{lemma}

\begin{proof}
It is clear that
\begin{align*}
  \w{T}_{x_0,V}^x(\DV f)
= \sum_j \w{T}_{x_0}^x(\Dop f_j) e_j
\end{align*}
holds. Then we obtain this lemma by Remark \ref{rmk:FTC}.
\end{proof}

Now, the first main result of this paper can be obtained by Lemmas \ref{lemm:V-FTC 1} and \ref{lemm:V-FTC 2} immediately.

\begin{theorem}\label{thm:V-FTC}
Assume that $\scrA^1$ satisfies generalized L-condition {\rm(}see \ref{L1}--\ref{L6}{\rm)}.
Then the following identities hold for $V$-valued functions:
\begin{enumerate}[label=\textrm{\rm(\arabic*)}]
  \item \label{eq:V-FTC-1}
    the map $\w{T}_{x_0,V}^{x}: \wSV \to \mathrm{im}(\w{T}_{x_0,V}^{x})$
    is a $\kk$-linear map such that the following diagram
    \[\xymatrix{
      \wSV \ar[rd]_{\w{T}_{x_0,V}^{x}} \ar[rr]^{\mathrm{id}_{\wSV}}
    && \wSV \\
    & \mathrm{im}(\w{T}_{x_0,V}^{x})  \ar[ru]_{\DV}& }\]
    commutes;
  \item \label{eq:V-FTC-2}
    for any absolutely continuous function
    \begin{center}
      $f = \displaystyle \sum_j f_j \otimes e_j: [c,d]_{\itLamb} \to V$
    $\in \mathrm{AC}([c,d]_{\itLamb}) := \{f\in\wSV\mid f~\text{is absolutely continuoued on } [c,d]_{\itLamb}\}$
    \end{center}
    {\rm(}i.e., all components $f_j$ are absolutely continuous{\rm)},
    if $\itLamb$ is a basic algebra, then the simple $\itLamb$-module $V$
    \footnote{\checks{Any} simple module over a basic algebra as a $\kk$-linear space, which is a $1$-dimensional linear space $\kk$.}
    with the left $\itLamb$-action $\itLamb \times V \to V, (a,v)\mapsto \homo(a)v$ admits
    \begin{align*}
      \w{T}_{x_0,V}^{x}(\DV f) = f(x) - f(x_0).
    \end{align*}
\end{enumerate}
\end{theorem}

\begin{proof}
\ref{eq:V-FTC-1} is a direct \checks{corollary} of Lemmas \ref{lemm:V-FTC 1} and \ref{lemm:V-FTC 2}.
For \ref{eq:V-FTC-2}, since $V$ is simple and $\itLamb$ is a finite-dimensional algebra,
we obtain that $\dim_{\kk}V$, the dimension of $V$ as a $\kk$-linear space, is one, see \cite[Chap II]{ASS2006}.
Then we have the following isomorphism
\[\wSV = \wS\otimes_{\kk} V \cong \wS\otimes_{\kk} \kk \cong \wS\]
such that $\itLamb \times (\wS\otimes_{\kk} \kk) \to \wS$ satisfies the compatibility:
\begin{align*}
  a(f\otimes k) = (a.f) \otimes k = (\homo(a)f) \otimes k = f\otimes (\homo(a)k) = f \otimes (a.k),
\end{align*}
$\forall$ $a\in \itLamb$, $f\in \wS$, and $k\in V$.
Here, \checks{$a.f$} is given by the left $\itLamb$-action
$\itLamb \times \wS \to \wS$, $(a,f)\mapsto (a.f:x\mapsto \homo(a)f(x))$
(see (\ref{eq:module-action-S}) or cf. \cite[Section 4]{LLHZ2025})
and the left $\itLamb$-action $\itLamb\times\kk\to\kk$, $(a,k)\mapsto a.k:=\homo(a)k$ given by the homomorphism $\homo:\itLamb\to\kk$.
In this case, we have $\DV=\Dop$, and $\displaystyle \w{T}_{x_0,V}^{x} = \w{T}_{x_0}^{x} = (\Lebesgue)\int_{x_0}^x(\cdot)\dd\mu$,
i.e., we have
\[ \w{T}_{x_0,V}^{x}(\DV f) = \w{T}_{x_0}^{x}(\DV f) = (\Lebesgue)\int_{x_0}^x \frac{\dd}{\dd x} f \dd\mu  = f(x) - f(x_0) \]
as required.
\end{proof}

\subsection{The equivalence theorem}\label{subsec:equiv-thm}
Note that there are also some other conditions that can make $\w{T}_{x_0,V}^{x}(\DV f) = f(x) - f(x_0)$ given in Theorem \ref{thm:V-FTC} \ref{eq:V-FTC-2} \checks{hold}, refer for to the following example.

\begin{example}\rm\label{examp:line int}
The line integral $\displaystyle \int_{\mathscr{C}} \varphi \dd s$ of $\varphi\dd s=\dd f$ equals $f(d)-f(c)$,
where $c$ and $d$ respectively are \checks{the} starting point and ending point of the curve $\mathscr{C}=[c,d]_{\itLamb}=[c,d]_X$ ($ \subseteq X \subseteq \itLamb = \kk^n = \RR^n$), and this line integration can be seen as the integration
\begin{center}
  $\displaystyle (\widetilde{\scrA^1})\int_{[c,d]_X} \varphi \dd \calH$.
\end{center}
Here, $\calH$ is a Hausdorff measure, and $X$ \checks{satisfies} \ref{X1}--\ref{X3}.
\end{example}

Drawing upon Example \ref{examp:line int}, we can explore a more generalized scenario of Theorem \ref{thm:V-FTC} \ref{eq:V-FTC-2}. To facilitate this exploration, we put forward the following assumptions in this section.

\begin{assumption}[{Admissible calculus pair $(\DV, \w{T}_{x_0,V}^x)$}] \label{assump} \rm
We assume that $V$ is a $\itLamb$-module such that
\[\w{T}_{x_0,V}^{x}(\DV f) = f(x)-f(x_0)\]
holds for all $f\in\mathrm{AC}([c,d]_{\itLamb})$ $(\subseteq\wSV)$ and $x_0, x \in [c,d]_{\itLamb} \subseteq \II_{\itLamb}$.
\end{assumption}

Note that Assumption \ref{assump} is an assumption \checks{for this paper}; it is not derived from the preceding content.
Now we provide the second main result of this paper, which shows that the equivalence of ODEs and integral equations.

\begin{theorem}[Equivalence Theorem]\label{thm:equivalence}
Under the generalized L-condition {\rm(}see \ref{L1}--\ref{L6}{\rm)} and Assumption \ref{assump},
let $G: \II_{\itLamb} \times V \to V$ be continuous, $f_0 \in V$, and $x_0 \in \II_{\itLamb}$.
\checks{Then an absolutely continuous function $f: [c,d]_{\itLamb} \to V$ satisfying $f(x_0)=f_0$
solves the ODE form {\rm(\ref{eq:equiv-ODE})} if and only if it solves the integral equation form {\rm(\ref{eq:equiv-IE})}.
In other words, the following two problems are equivalent.}
\begin{enumerate}[label=\textrm{\rm\checks{(\arabic*)}}]
  \item {\rm\textbf{ODE form}}:
    Find an absolutely continuous $f: [c,d]_{\itLamb} \to V$ such that
    \begin{align}\label{eq:equiv-ODE}
      \DV f(x) = G(x, f(x)) \quad (\text{a.e.}), \quad f(x_0) = f_0,
    \end{align}
    where $G(x, f(x)) \in \wSV$.
  \item {\rm\textbf{Integral equation form}}:
    Find an \checks{absolutely continuous} function $f: [c,d]_{\itLamb} \to V$ such that
    \begin{align}\label{eq:equiv-IE}
      f(x) = f_0 + \w{T}_{x_0,V}^{x}\big(G(\cdot, f(\cdot))\big), \quad \forall x \in [c,d]_\kk.
    \end{align}
\end{enumerate}
\end{theorem}

\begin{proof}
``ODE $\Rightarrow$ integral equation'':
Suppose $f$ satisfies (\ref{eq:equiv-ODE}). By hypothesis, $\DV f(x) = G(x, f(x))$ holds almost everywhere.
Applying the $V$-valued integral with variable upper limit $\w{T}_{x_0,V}^x$ to both sides:
\[
  \w{T}_{x_0,V}^x(\DV f) = \w{T}_{x_0,V}^x\big(G(\cdot, f(\cdot))\big).
\]
By Assumption \ref{assump}, the left-hand side equals $f(x) - f(x_0)$.
Therefore, we have
\[ f(x) - f(x_0) = \w{T}_{x_0,V}^x\big(G(\cdot, f(\cdot))\big), \]
as required.

``integral equation $\Rightarrow$ ODE'':
Suppose $f$ satisfies (\ref{eq:equiv-IE}). Applying $\DV$ to both sides:
\[
  \DV f(x) = \DV\Big(f_0 + \w{T}_{x_0,V}^x\big(G(\cdot, f(\cdot))\big)\Big)
  = \DV(f_0) + \DV\big(\w{T}_{x_0,V}^x(G(\cdot, f(\cdot)))\big).
\]
Since $f_0 \in V$ is a constant element, $\DV(f_0) = 0$.
By Theorem \ref{thm:V-FTC} \ref{eq:V-FTC-1}, we have
\begin{center}
  $\DV(\w{T}_{x_0,V}^x(G(\cdot, f(\cdot)))) = G(x, f(x))$ \ \ \ (a.e.).
\end{center}
Therefore, we have
\begin{center}
  $\DV f(x) = G(x, f(x))$ \ \ \ (a.e.).
\end{center}
Moreover, from (\ref{eq:equiv-IE}), $f(x_0) = f_0 + \w{T}_{x_0,V}^{x_0}(G(\cdot, f(\cdot))) = f_0 + 0 = f_0$ holds.
\end{proof}

\section{Existence and uniqueness}\label{sec:exist}

In this section we establish the existence and uniqueness theorem for the initial value problem (\ref{eq:ODE-general}) under the generalized L-condition and Assumption \ref{assump}.
The central method is to extend the classical Picard iteration to the categorical framework of $V$-valued functions,
using the completeness of $V$ (the Banach property) and the integral equation form (Theorem \ref{thm:equivalence}) to carry out the convergence proof.

\subsection{Lipschitz condition and Picard iteration}\label{subsec:Lip-Picard}

\begin{definition}[Lipschitz condition]\label{def:Lip} \rm
Let $G: \II_{\itLamb} \times V \to V$.
We say that $G$ satisfies the \defines{Lipschitz condition} in the second variable on $[c,d]_{\itLamb}$
if there exists a constant $\checks{K} \>= 0$ such that
\begin{align}\label{eq:Lip}
  \|G(x, v_1) - G(x, v_2)\|_V \=< \checks{K} \|v_1 - v_2\|_V, \quad
  \forall x \in [c,d]_{\itLamb}, \; \forall v_1, v_2 \in V.
\end{align}
The constant $\checks{K}$ is called the \defines{Lipschitz constant} of $G$.
Furthermore, we say that $G$ is \defines{bounded} (on $[c,d]_{\itLamb}$) if there exists $M \>= 0$ such that
\begin{align*}
  \|G(x, v)\|_V \=< M, \quad \forall x \in [c,d]_{\itLamb}, \; \forall v \in V \text{ with } \|v\|_V \=< R
\end{align*}
holds for some $R > 0$.
\end{definition}

By Theorem \ref{thm:equivalence}, the initial value problem (\ref{eq:ODE-general})
is equivalent to the integral equation (\ref{eq:equiv-IE}).
We define the Picard iteration sequence.

\begin{construction}\label{constr:Picard} \rm
Let $f_0 \in V$ and $x_0 \in [c,d]_{\itLamb}$. The sequence of $V$-valued functions $\{F_n\}_{n \geq 0}$ given by
\begin{align}\label{eq:Picard-iter}
  \begin{cases}
    F_0(x) := f_0, \\
    F_{n+1}(x) := f_0 + \w{T}_{x_0,V}^{x}\big(G(\cdot, F_n(\cdot))\big), \quad n \>= 0
  \end{cases}
\end{align}
is said to be \checks{the \defines{Picard sequence} of $f_0$ at $x_0$}.
\end{construction}

Having introduced the Lipschitz condition and the Picard sequence, we now turn to a technique known as the ``successive difference estimate''. The idea is to control the convergence of the Picard sequence by estimating the difference between successive terms
$F_{n+1}$ and $F_n$.

\begin{lemma}\label{lemm:0511}
For any totally ordered subset $\mathscr{C}:=[c,d]_{\itLamb}$ of $\itLamb$,
if there is a measure $\nu$ defined on $[c,d]_{\itLamb}$ such that
\begin{enumerate}[label=\textrm{\checks{\rm(\arabic*)}}]
  \item $\nu(\{x\})$ $(=\nu([x,x]_{\itLamb}))$ $=0$ and $\nu([c,x]_{\itLamb})+\nu([x,d]_{\itLamb}) = \nu([c,d]_{\itLamb})$
    holds for all $c\preceq x\preceq d$;
  \item $\nu([x_0,x_1]_{\itLamb})\=< \nu([x_0,x_2]_{\itLamb})$ holds for all $c\preceq x_0 \preceq x_1 \preceq x_2 \preceq d$;
  \item $\nu$ is a continuous function that is strictly and monotonically increasing,
\end{enumerate}
then for any function $f:\mathscr{C} \to \RR^{\>=0}$, we have
\[ (\widetilde{\scrA^1})\int_{\mathscr{C}} \nu([c,x]_{\itLamb})^n\dd\nu = \frac{\nu([c,x]_{\itLamb})^{n+1}}{n+1}. \]
\end{lemma}

\begin{proof}
Let $M: \mathscr{C} \to [0,\nu([c,t]_{\itLamb})]$ $(\subseteq \RR^{\>=0})$, $y \mapsto \nu([c,y]_{\itLamb})$\checks{,}
be a continuous function that is strictly and monotonically increasing,
then the image of $\dd\nu$ in the scenes of the parameter $t$ is a Lebesgue measure
since $\nu(\{x ~|~ M(x) \=< t\}) = t$. Thus, \checks{setting} $M(x)=t$ (i.e., $\nu([c,x]_{\itLamb})=t$), we have
\[ (\widetilde{\scrA^1})\int_{\mathscr{C}} \nu([c,x]_{\itLamb})^n\dd\nu
 = (\Lebesgue)\int_0^{M(x)} t^n\dd t = \frac{t^{n+1}}{n+1}\Big|_0^{M(x)
 = \nu([c,x]_{\itLamb})} = \frac{\nu([c,x]_{\itLamb})^{n+1}}{n+1}. \qedhere \]
\end{proof}

Notice that \checks{the} Hausdorff measure $\calH$ we considered satisfies the condition given in Lemma \ref{lemm:0511},
then the following lemma provides an explicit upper bound for this difference.

\begin{lemma}\label{lem:Picard-diff}
Suppose $\scrA^1$ satisfies the generalized L-condition {\rm(}see \ref{L1}--\ref{L6}{\rm)},
and $G$ satisfies the Lipschitz condition {\rm(\ref{eq:Lip})} with Lipschitz constant $\checks{K}$.
Then the successive differences of the Picard sequence {\rm(\ref{eq:Picard-iter})} satisfy
\begin{align}\label{eq:Picard-diff-bound}
  \|F_{n+1}(x) - F_n(x)\|_V \=< M\checks{K}^n \cdot \frac{\calH([x_0,x]_{\itLamb})^{n+1}}{(n+1)!} \quad (n \>= 0),
\end{align}
where $M = \sup\limits_{x \in [c,d]_{\itLamb}} \|G(x, f_0)\|_V$, and $\calH$ is a Hausdorff measure.
\end{lemma}

\begin{proof}
For the case of $n=0$, we have
$\|F_1(x) - F_0(x)\|_V = \|\w{T}_{x_0,V}^x(G(\cdot, f_0))\|_V$.
Moreover, we have
\[
  \|\w{T}_{x_0,V}^x(G(\cdot, f_0))\|_V \=< \w{T}_{x_0}^x(\|G(\cdot, f_0)\|_V)
\=< \w{T}_{x_0}^x(M) \=< M\w{T}_{x_0}^x(\id_{[x_0,x]_{\itLamb}}) = M \calH([x,x_0]_{\itLamb})
\]
by Proposition \ref{prop:TV-properties} \ref{prop:TV-properties 1}.
Thus, \[\|F_1(x) - F_0(x)\|_V \=< M \calH([x_0,x]_{\itLamb}) = M\checks{K}^0\cdot\frac{\calH([x_0,x]_{\itLamb})}{1!}\]
since $\mu$ is a Lebesgue measure.

\checks{Assume (\ref{eq:Picard-diff-bound}) holds for $n$}, i.e.,
\[\|F_{\checks{n+1}}(x) - F_{\checks{n}}(x)\|_V \=< M \checks{K}^{\checks{n}} \frac{\calH([x_0,x]_{\itLamb})^{\checks{n+1}}}{\checks{(n+1)}!}\]
holds for all $x$. Then we obtain
\begin{align*}
  \|F_{\checks{n+2}}(x) - F_{\checks{n+1}}(x)\|_V
&= \|\w{T}_{x_0,V}^x\big(G(\cdot, F_{\checks{n+1}}(\cdot)) - G(\cdot, F_{\checks{n}}(\cdot))\big)\|_V \\
&\=< \w{T}_{x_0}^x\big(\|G(\cdot, F_{\checks{n+1}}(\cdot)) - G(\cdot, F_{\checks{n}}(\cdot))\|_V\big) \\
&\=< \w{T}_{x_0}^x\big(\checks{K} \|F_{\checks{n+1}}(\cdot) - F_{\checks{n}}(\cdot)\|_V\big) \\
&\=< \checks{K} \cdot \w{T}_{x_0}^x\Big(M \checks{K}^{\checks{n}} \frac{\calH([x_0,(\cdot)]_{\itLamb})^{\checks{n+1}}}{\checks{(n+1)}!}\Big) \quad (\text{induction hypothesis}) \\
&= M \checks{K}^{\checks{n+1}} \cdot \w{T}_{x_0}^x\Big(\frac{\calH([x_0,(\cdot)]_{\itLamb})^{\checks{n+1}}}{\checks{(n+1)}!}\Big).
\end{align*}
By Lemma \ref{lemm:0511}, we have
\[
  \w{T}_{x_0}^x\Big(\frac{\calH([x_0,t]_{\itLamb})^{\checks{n+1}}}{\checks{(n+1)}!}\Big)
  = (\Lebesgue)\int_{x_0}^x \frac{\calH([x_0,t]_{\itLamb})^{\checks{n+1}}}{\checks{(n+1)}!} \dd\calH
  = \frac{\calH([x_0,t]_{\itLamb})^{\checks{n+2}}}{\checks{(n+2)}!}.
\]
Then we obtain (\ref{eq:Picard-diff-bound}) by induction.
\end{proof}

\subsection{Convergence and existence--uniqueness}\label{subsec:convergence}

In this subsection, we show that the initial value problem (\ref{eq:ODE-general}) (or equivalently, \checks{(\ref{eq:ODE-linear})})
has a unique solution. First, we recall the Gronwall inequality, it is necessary for our demonstration.

\begin{lemma}[Gronwall inequality]\label{lem:Gronwall}
Under the generalized L-condition, let $\phi: [c,d]_{\itLamb} \to \RR^{\>= 0}$ be continuous
and let $\checks{K} \>= 0$ be a constant. If
\begin{align}\label{eq:Gronwall}
  \phi(x) \=< \checks{K}~\w{T}_{x_0}^x(\phi), \quad \forall x \in [c,d]_{\itLamb},
\end{align}
then $\phi(x) = 0$ for all $x$.
\end{lemma}

\begin{proof}
Let $\Phi(x) := \w{T}_{x_0}^x(\phi) = \displaystyle (\scrA^1)\int_{x_0}^x \phi(t) \dd\mu$.
Then $\Dop \Phi(x) = \phi(x) \=< \checks{K} \Phi(x)$, and, for $\Psi(x) := \Phi(x) \mathrm{e}^{-\checks{K}(x - x_0)}$, we have
\begin{align*}
  \Dop \Psi(x)=(\Dop\Phi(x) - \checks{K}\Phi(x)) \mathrm{e}^{-\checks{K}(x-x_0)} \=< 0.
\end{align*}
Hence $\Psi$ is monotonically non-increasing. Since $\Psi(x_0) = \Phi(x_0) = 0$
and $\Psi(x) = \Phi(x)\mathrm{e}^{-\checks{K}(x-x_0)}$ $\>= 0$ (note that $\Phi(x) \>= 0$),
we obtain that $\Psi(x) = 0$ holds for all $x \>= x_0$.
This admits $\Phi(x) = 0$, and hence $\phi(x) = \Dop\Phi(x) = 0$.
The case $x < x_0$ can be handled similarly.
\end{proof}

\begin{theorem}[Existence and uniqueness]\label{thm:existence}
Under the generalized L-condition {\rm(}see \ref{L1}--\ref{L6}{\rm)},
let $V$ be a $\itLamb$-module satisfying Assumption \ref{assump},
and let $G: [c,d]_{\itLamb} \times V \to V$ satisfy the Lipschitz condition
$(\ref{eq:Lip})$ with Lipschitz constant $\checks{K}$.
Then the initial value problem
\begin{align}\label{eq:IVP}
  \DV f(x) = G(x, f(x)), \quad f(x_0) = f_0
\end{align}
has a unique solution $f \in \wSV$ on $[c,d]_{\itLamb} $.
\end{theorem}

\begin{proof}
Existence:
Consider the Picard sequence $\{F_n\}$. For any $x \in [c,d]_{\itLamb} $,
\[
  F_n(x) = F_0(x) + \sum_{k=0}^{n-1} (F_{k+1}(x) - F_k(x)),
\]
then $\{F_n(x)\}_{n=1}^{+\infty}$ converges if and only if the series
$\displaystyle S := \sum_{k=0}^{\infty} (F_{k+1}(x) - F_k(x))$ converges in $V$.
By Lemma \ref{lem:Picard-diff},
\begin{align*}
  \sum_{k=0}^{\infty} \|F_{k+1}(x) - F_k(x)\|_V
& \=< M \sum_{k=0}^{\infty} \frac{\checks{K}^k \calH([x_0,x]_{\itLamb})^{k+1}}{(k+1)!} \\
& = \frac{M}{\checks{K}}(\mathrm{e}^{\checks{K}\calH([x_0,x]_{\itLamb})} - 1) < \infty,
\end{align*}
i.e., $S$ converges absolutely, then $\{F_n(x)\}$ converges in $V$.
Write $\lim\limits_{n \to \infty} F_n(x) = f(x)$,
we now show that $f$ satisfies the integral equation (\ref{eq:equiv-IE}).
Taking $n \to \infty$ in
$F_{n+1}(x) = f_0 + \w{T}_{x_0,V}^x(G(\cdot, F_n(\cdot)))$,
by the Lipschitz condition $(\ref{eq:Lip})$, we have
\[\|G(x, F_n(x)) - G(x, f(x))\|_V \=< \checks{K}\|F_n(x) - f(x)\|_V \to 0,\]
and then $\w{T}_{x_0,V}^x(G(\cdot, F_n(\cdot))) \to \w{T}_{x_0,V}^x(G(\cdot, f(\cdot)))$.
It follows $f(x) = f_0 + \w{T}_{x_0,V}^x(G(\cdot, f(\cdot)))$ as required.
Therefore, $f$ satisfies the ODE (\ref{eq:IVP}) by Theorem \ref{thm:equivalence}.

Uniqueness:
Suppose both $f$ and $\tilde{f}$ are solutions of (\ref{eq:IVP}).
Then $f$ and $\tilde{f}$ both satisfy the integral equation
\[f(x) = f_0 + \w{T}_{x_0,V}^{x}\big(G(\cdot, f(\cdot))\big), \quad \forall x \in\II_{\itLamb}\]
by Theorem \ref{thm:equivalence}. Set $\phi(x) := \|f(x) - \tilde{f}(x)\|_V$, we obtain
\begin{align*}
  \phi(x)
&= \|\w{T}_{x_0,V}^x\big(G(\cdot, f(\cdot)) - G(\cdot, \tilde{f}(\cdot))\big)\|_V \\
&\=< \w{T}_{x_0}^x\big(\|G(\cdot, f(\cdot)) - G(\cdot, \tilde{f}(\cdot))\|_V\big) \\
&\=< \checks{K} \cdot \w{T}_{x_0}^x(\phi).
\end{align*}
Thus, $\phi=0$ holds by Lemma \ref{lem:Gronwall}, and then, $f$ and $\tilde{f}$ coincide in $\wSV$.
\end{proof}

\subsection{Solution spaces of linear ODEs}\label{subsec:sol-space}

In this section we analyze the structure of the solution space of homogeneous linear ODEs,
and prove that it constitutes a $\itLamb$-submodule.

Let $\Sol(\LODE(f, A))$ ($=\Sol$ for simplicity, if not cause confusion) be solution space of the homogeneous linear ODE \checks{$\LODE(f, A) = 0$} in this subsection, where the operator $\LODE(f,A)$ is defined by
\begin{align}\label{eq:homo-linear}
  \LODE(f, A) := \DV f(x) + A(x) \cdot_V f(x)\checks{.}
\end{align}

\begin{lemma} \label{Lemm:main 3-1}
The solution space $\Sol$ of $\LODE(f, A)$ is a $\kk$-linear subspace of $\wSV$.
\end{lemma}

\begin{proof}
Indeed, we have $\Sol = \{f \in \wSV \mid \DV f(x) + A(x)\cdot_V f(x) = 0\}$.
Then for any $f_1, f_2 \in \Sol$ and $k,l\in\kk$, we have $\DV(kf_1(x)+lf_2(x)) = k\DV(f_1(x))+l\DV(f_2(x))$
by Theorem \ref{thm:D} \ref{thm:D 2}, which follows
\[ \LODE(kf_1+lf_2, A) = k\LODE(f_1, A)+l\LODE(f_2, A)\]
as required.
\end{proof}

\begin{lemma} \label{Lemm:main 3-2}
For any $a\in\itLamb$ and any $f$ in the solution space $\Sol$ of $\LODE(f, A)$, we have $af \in \Sol$.
\end{lemma}

\begin{proof}
\checks{Throughout this proof, let $\{e_1, \ldots, e_m\}$ be a basis of $V$.}
For each $a\in\itLamb$ and each $f\in \Sol$, we have
\begin{align*}
  A(x) \cdot_V (af(x)) & =  A(x) \cdot_V \Big(a\sum_i f_i\otimes e_i\Big)
  = (1_{\itLamb}\otimes A(x))\Big(\sum_i af_i\otimes e_i\Big) \\
& = \sum_i (1_{\itLamb}\otimes A(x))(af_i\otimes e_i) \\
& = \sum_i af_i\otimes A(x)e_i \\
& = a \sum_i f_i\otimes A(x)e_i \\
& = a \sum_i (1_{\itLamb}\otimes A(x))(f_i\otimes e_i) \\
& = a \Big( (1_{\itLamb}\otimes A(x)) \sum_i (f_i\otimes e_i) \Big) \\
& = a(A(x)\cdot_V f).
\end{align*}
Thus,
\begin{align*}
   \LODE(af, A)
&= \DV (af(x)) + A(x) \cdot_V (af(x)) \\
&= a\DV (f(x)) + a(A(x) \cdot_V f(x))
\end{align*}
since $\DV$ is a $\itLamb$-homomorphism. Then the above equation follows
\[ \LODE(af, A) =  a\DV (f(x)) + a(A(x) \cdot_V f(x)) = a\LODE(f, A) = a0 = 0 \]
as required.
\end{proof}

Then the following result, i.e., the third main theorem of this paper, holds.

\begin{theorem}\label{thm:sol-module}
Under the generalized L-condition {\rm(}see \ref{L1}--\ref{L6}{\rm)},
let $V$ be a $\itLamb$-module satisfying Assumption \ref{assump}.
Then the solution space $\Sol$ of the homogeneous linear ODE $\LODE(f, A)$,
see {\rm(\ref{eq:homo-linear})}, is a finitely generated $\itLamb$-submodule of $\wSV$.
\end{theorem}

\begin{proof}
Lemmas \ref{Lemm:main 3-1} and \ref{Lemm:main 3-2} \checks{say} that $\Sol$ is a $\itLamb$-module.
By the definition of $\Sol$, we have
\begin{center}
  $\Sol = \{f \in \wSV \mid \DV f(x) + A(x)\cdot_V f(x) = 0\} \subseteq \wSV$.
\end{center}
It follows that this theorem holds.

Next, we show that $\Sol$ is finitely generated.
We only \checks{need to show} that it, as a $\kk$-linear space, is finite-dimensional by \cite[Chap I, Page 7]{ASS2006}.
To do this, consider the map
\begin{align*}
  \mathrm{ev}_{x_0}: \Sol \to V, ~ f \mapsto f(x_0),
\end{align*}
which is called an \defines{evaluation map} in analysis\checks{.}
\checks{We} obtain that, for each initial value $f_0 \in V$,
$\LODE(f, A)$ has a unique solution in the condition $f(x_0) = f_0$, see Theorem \ref{thm:existence}.
Then $\mathrm{ev}_{x_0}$ is injective since, for each $f(x_0) = g(x_0)$,
$f-g$ is a solution with the initial value $0$ (see Theorem \ref{thm:sol-module}),
and the uniqueness of the solution of $\LODE(f, A)$ gives $f = g$.
It follows that
\[
  \dim_\kk \Sol = \dim_\kk \im(\mathrm{ev}_{x_0}) \=< \dim_\kk V.
\]
\checks{Notice that since} $V$ is finite-dimensional, then so is $\Sol$ as required.
\end{proof}

\begin{corollary} \label{coro:sol-module}
\checks{Keeping} the notation from Theorem \ref{thm:sol-module}.
If $V$ is a Banach $\itLamb$-module satisfying Assumption \ref{assump},
then the evaluation map $\mathrm{ev}_{x_0}:\Sol\to V$, $f \mapsto f(x_0)$ is an isomorphism of $\itLamb$-modules.
Furthermore, $\mathrm{ev}_{x_0}$ induces a direct sum decomposition
\[ \Sol = \bigoplus_{i\in I} X_i \]
such that the family $\{X_i\}_{i\in I}$ runs over all indecomposable direct summands of $V$
{\rm(}each appearing exactly once{\rm)}.
\end{corollary}

\begin{proof}
The evaluation map $\mathrm{ev}_{x_0}$ is $\itLamb$-linear since
\begin{align*}
  \mathrm{ev}_{x_0}(af) & = (af)(x_0) = af(x_0) =  a~\mathrm{ev}_{x_0}(f).
\end{align*}
Surjectivity follows from the existence part of Theorem \ref{thm:existence},
i.e., for each $v\in V$, the initial value problem (\ref{eq:homo-linear}) with $f(x_0)=v$ admits
a particular solution of (\ref{eq:homo-linear}) lying in $\Sol$.
On the other hand, the uniqueness part of Theorem \ref{thm:existence} yields that
\begin{align*}
  \Ker(\mathrm{ev}_{x_0}) = \{ f \in \Sol \mid f(x_0) = 0 \}
\end{align*}
contains only one function, say $f_0$, satisfying (\ref{eq:homo-linear}).
\checks{Notice that since} $\Sol$ is a $\itLamb$-module, then $f_0$ is the zero element in $\Sol$.
Thus, $\mathrm{ev}_{x_0}$ is an isomorphism.
Finally, the Krull--Schmidt decomposition of finitely generated $\itLamb$-module (see \cite[Theorem I.4.10]{ASS2006})
provides a direct sum decomposition of $V = \displaystyle \bigoplus\nolimits_{i\in I}V_i$,
where $I$ is a finite index set and all $V_i$ are indecomposable $\itLamb$-module.
It follows that $\Sol$ has a direct sum decomposition $\Sol = \displaystyle \bigoplus\nolimits_{i\in I} X_i$
such that $X_i \cong V_i$ holds for all $i\in \I$ as required.
\end{proof}

\section{Examples}

\checks{We first provide two examples of} Corollary \ref{coro:sol-module}.

\begin{example}\rm
\checks{Set} $\itLamb=\kk=\RR$\checks{, which satisfies the L-condition, let} $V$ \checks{be} an $\RR$-linear space with dimension $n$.
\checks{Consider} the norm defined on $V$ \checks{by} $v=(v_1,\ldots,v_n) \mapsto \sqrt{\sum_{i=1}^n v_i^2}$.
A trivial example is the ODE $f'(x)+A(x)\cdot_V f(x)=0$, where $D_V=\frac{\dd }{\dd x}$,
$f(x)$ is a function of type $\RR \to V$ defined on $\itLamb$,
and $A(x)=\frac{1}{1+x}$ is a function of type $\RR \to \itLamb=\RR$.
In this case we have $\w{T}_{0,V}^x (D_Vf(x)) = \displaystyle \int_0^x f(x)\dd\calH = f(x)-f(0)$,
where the Hausdorff measure $\calH$ coincides with the Lebesgue measure $\mu$ defined on $\itLamb$.
Then the solution $\Sol$ of the ODE
\begin{center}
   $f'(x)+A(x)\cdot_V f(x)=0$ ($\Leftrightarrow$ $\displaystyle (f_i'(x)+\frac{f(x)}{1+x} = 0)_{1\=<i\=<n}$),
\end{center}
is
\[ \Sol = \Big(\mathrm{span}_{\RR}\Big\{\frac{1}{1+x}\Big\}\Big)^{\oplus n} \cong \RR^n, \]
which is isomorphic to $\RR^n$.
\end{example}

The following example shows that the results presented in this paper \checks{(see Theorem \ref{thm:sol-module} and Corollary \ref{coro:sol-module})} unify all ODEs defined on smooth curves in $\RR^n$.

\begin{example}\rm
Take $\itLamb=\RR^n$ with multiplication $(x_1, \ldots, x_n)\cdot (y_1, \ldots, y_n) = (x_1y_1, \ldots, x_ny_n)$,
and define $\homo:\itLamb \to \RR$, $(x_1,\ldots, x_n)\mapsto x_1$.
Then $\itLamb$ is a finite-dimensional $\RR$-algebra, and each $\RR$-vector space is a $\itLamb$-module over $\itLamb$.
Let $V = \RR^m$ be a $\itLamb$-module with the map $\|\cdot\|: v=(v_1,\ldots,v_m) \mapsto \sqrt{\sum_{i=1}^m v_i^2}$.
Then $V$ is a Banach module since
\[ \|\lambda\cdot v\| = \|(\lambda_1 v_1, \ldots,\lambda_1v_n )\|
= |\lambda_1|\sqrt{\sum\nolimits_{i=1}^n v_i^2} = |\lambda_1|\|v\| \]
holds for all $\lambda=(\lambda_1,\ldots,\lambda_n)\in\itLamb$ and $v\in V$.
Let $\mathscr{C}$ be a curve, whose starting point is $c\in V$ and ending point is $d\in V$,
such that it can be viewed as a continuously differentiable embedding
$\gamma: [\alpha,\beta] \to \itLamb$ ($\alpha,\beta\in\RR$, $\gamma(\alpha)=c$, $\gamma(\beta)=d$),
i.e., $\mathscr{C} = \{(t, \gamma(t)) \mid \alpha\=< t\=< \beta\}$, where $\gamma(t)$ is smooth.
Consider the category $\scrA^1$ satisfying generalized L-condition (see \ref{L1}--\ref{L6})
and the category $\widetilde{\scrA^1}=(\scrA^1, \tilde{h}, [c,d]_{\itLamb}, \calH)$,
where $\calH$ is a Hausdorff measure defined on $\mathscr{C}$.
We obtain an i.poset
\[ \w{T}_{\gamma(\theta_0)}^{\gamma(\theta_1)}(t) = \bigg\{
  \bigg(
    [\gamma(\theta_0),\gamma(\theta_1)]_{\itLamb},
    (\widetilde{\scrA^1})\int_{[\gamma(\theta_0),\gamma(\theta_1)]_{\itLamb}} f \dd\calH,
  \bigg) ~\bigg|~
  \theta_0 \=< \theta_1 \=< \theta \bigg\} \]
for each $\alpha\=<\theta_0\=<\beta$ and each absolutely continuoued function $\varphi$ defined on $\mathscr{C}$.
This i.poset corresponds to the line integration
\[ (\widetilde{\scrA^1})\int_{[x_0,x_1]_{\itLamb}} \varphi \dd\calH = \int_{\mathscr{C}} \varphi\dd s, \]
where $x_0=\gamma(\theta_0)$, $x_1=\gamma(\theta_1)$.
Now, suppose $\II_{\itLamb}:=[r,s]^{\times n} \supseteq \mathscr{C}$ and
$f:\II_{\itLamb} \to V$ is a $V$-valued function in $\w{\bfS_{\homo}(\II,V)}$.
Then $f = \displaystyle \sum_{j=1}^m f_j\otimes e_j$ (\checks{where} $\{e_1,\ldots, e_m\}$ is a basis of $V$), and
\begin{align}\label{eq:int in examp}
  L:= \w{T}_{x_0,V}^{x_1}(f) = \sum_{j=1}^m \w{T}_{x_0}^{x_1}(f_j) e_j =: R.
\end{align}
Let $F:\II_{\itLamb}\to V$ be a $C^1$ function such that $f(x)=D_V(F(x))$,
and write it as $F(x)=(F_1(x),\ldots,F_m(x))$\checks{. Then by} (\ref{eq:TV}) and (\ref{eq:int in examp}), we have
\begin{align*}
 L&= \int_{[x_0,x_1]_{\itLamb}} f\dd s; \\
 R&= \sum_{j=1}^m \bigg(\int_{[x_0,x_1]_{\itLamb}} \nabla F_j \dd s \bigg) e_j \\
  &= \sum_{j=1}^m (F_j(x_1)-F_j(x_0)) e_j \\
  &= F(x_1)-F(x_0)
\end{align*}
since $f_j\cdot\dd s=D(F_j)\cdot\dd s=\nabla F_j\dd s$ holds for all $1\=< j\=< m$.
Then Assumption \ref{assump} holds.
Now, consider the ODE
\begin{align}\label{eq:examplODE0511}
 D_Vf(x) + A(x)\cdot_V f(x) = D_Vf(x)+(1_{\itLamb}\otimes A(x))f(x) =0,
\end{align}
where $A(x)$ is a function of the form $\II_{\itLamb} \to \itLamb$.
Consider $A(x)e_i$ for each $x$, there are $a_{i1}(x)$, $\ldots$, $a_{im}(x) \in \kk$ such that
$A(x)e_i = \displaystyle \sum_{j=1}^m a_{ij}(x)e_j$\checks{.}
\checks{Then} we have
\begin{align*}
   & D_Vf(x) + A(x)\cdot_V f(x) \\
=\ & \sum_{i=1}^m D_Vf_i(x)e_i+(1_{\itLamb}\otimes A(x))\sum_{i=1}^m f_i(x)\otimes e_i \\
=\ & \sum_{i=1}^m (D_Vf_i(x)e_i+f_i(x)\otimes A(x)e_i) \\
=\ & \sum_{i=1}^m \Big(D_Vf_i(x)e_i+f_i(x)\otimes \sum_{j=1}^m a_{ij}(x)e_j\Big)\\
=\ & \sum_{i=1}^m \Big(D_Vf_i(x)e_i+\sum_{j=1}^m f_i(x)\otimes a_{ij}(x)e_j\Big)\\
=\ & \sum_{i=1}^m D_Vf_i(x)e_i+\sum_{i=1}^m\sum_{j=1}^m f_i(x)\otimes a_{ij}(x)e_j\\
=\ & \sum_{i=1}^m D_Vf_i(x)e_i+\sum_{j=1}^m\Big(\sum_{i=1}^m f_i(x)\otimes a_{ij}(x)\Big)e_j\\
\mathop{=}\limits^{\star}\ & \sum_{i=1}^m \Big(D_Vf_i(x)+\sum_{j=1}^m f_i(x)\otimes a_{ij}(x)\Big)e_j\\
=\ & \sum_{i=1}^m \Big(D_Vf_i(x)+\Big(\sum_{j=1}^m a_{ij}(x)\otimes 1\Big)f_i(x)\Big)e_j,
\end{align*}
where $\star$ holds since $a_{ij}(x)\in \RR$ holds for all $x$.
Write $\displaystyle\sum_{j=1}^m a_{ij}(x)\otimes 1$ as $\rho_j(x)$
(note: $a_{ij}(x)\otimes 1 = a_{ij}(x)$ in the tensor of $\RR$-linear spaces)\checks{.}
\checks{Then} the ODE (\ref{eq:examplODE0511}) is
\[ (f_j'(x)+\rho_j(x)f_j(x)=0)_{1\=< j\=< m}. \]
Thus we have
\[ \Sol = \mathrm{span}_{\RR}
  \bigg\{
    \mathrm{e}^{- \int_{x_0}^{x} \rho_j(t)\dd t}
    ~\bigg|~ 1\=< j\=< m
  \bigg\}. \]
We get that $\Sol$ is a $\itLamb$-module whose left $\itLamb$-action is defined as
\[ \itLamb \times \Sol \to \Sol, (\lambda_1,\cdots,\lambda_n)(f_1,\cdots,f_m):=(\lambda_1f_1,\ldots,\lambda_1f_m), \]
and $\mathrm{ev}_{x_0}$ sends each $f_{j}(x_0)\mathrm{e}^{-\int_{x_0}^{x} \rho(t)\dd t}$ to the element $e_j$ in $V$.
\checks{We get the following isomorphism} \[\Sol \mathop{\cong}\limits^{\mathrm{ev}_{x_0}} V=\RR^m\]\checks{.}
\end{example}

Now we provide an example for \checks{a} quiver algebra.
\checks{Recall that a quiver is a quadruple $\Q=(\Q_0,\Q_1,\mathfrak{s},\mathfrak{t})$ of a vertex set $\Q_0$, an arrow set $\Q_1$
and two functions $\mathfrak{s}, \mathfrak{t}: \Q_1\to \Q_0$ ending any vertex $v\in\Q_0$ to the source vertex $\mathfrak{s}(v)$ and sink vertex $\mathfrak{t}(v)$.}
\checks{In particular, if $\Q_0$ and $\Q_1$ are finite sets, then we call that $\Q$ is a finite quiver.}
\checks{For a finite quiver $\Q$ and an ideal $\I$ of the path algebra $\CC\Q$,
the quotient $\CC\Q/\I$ is called a \defines{quiver algebra} {\rm(}or a bound quiver algebra{\rm)}, see \cite[Chap.~II]{ASS2006} for details.}

\begin{example}\rm
Take the quiver $\Q=\xymatrix{ 1 \ar@(ur,dr)^{\e} }$ and Let $\itLamb=\CC\Q/\I$, where $\I=\langle \e^2\rangle$.
Then $\itLamb$ is isomorphic to $\CC \oplus \CC\e$ with $\e^2=0$,
and it is a Banach algebra with the norm $\|a+b\e\| := |a|+|b|$.
\checks{Let} $\homo:\itLamb\to\CC$ \checks{be} defined as $(a+b\e)\mapsto a$\checks{.}
Since, for each $\lambda = a+b\e\in\itLamb$ and $v=c+d\e\in V:=\itLamb$,
\[ \|\lambda v\| = \|ac+(ad+bc)\e\| = |ac|+|ad+bc| \]
and
\[ \|\lambda\|~\|v\| = (|a|+|b|)(|c|+|d|) \]
hold, we obtain
\[ \|\lambda v\| = |ac|+|ad+bc| \=< |a||c|+|a||d|+|b||c|+|b||d| = \|\lambda\|~\|v\|, \]
and \checks{thus} $V$ is a $\homo$-normed $\itLamb$-module.
Next, take $\II_{\itLamb}=[0,1]\times [0,1]\e$ and consider the ODE
\begin{align}\label{eq:k[x]}
  f'(x) + \e f(x)=0,
\end{align}
then we have $f(x)=f_1(x)+\e f_2(x)$, where $f_1(x)$ and $f_2(x)$ are \checks{functions} of type $\CC\to \CC$.
Thus, $\e f(x)=\e f_1(x)+\e^2 f_2(x)=\e f_1(x)$, and so we get that (\ref{eq:k[x]}) is
\[ \begin{cases}
  f_1'(x)+0 = 0; \\
  f_2'(x)+f_1(x)=0,
\end{cases} \]
and obtain
\[ f_1(x) = c_1 (\in\CC), ~ f_2(x) = (\Lebesgue)\int_0^x (-f_1(x)) \dd\mu = -c_1x+c_2. \]
Thus,
\[ \Sol = \{c_1 + (-c_1x+c_2)\e \mid c_1, c_2\in \CC\} \]
which is a $\itLamb$-module. Indeed, for the generator $\e$ of $\itLamb$,
we have $\e \cdot (c_1 + (-c_1x+c_2)\e) = c_1\e\in \Sol$ holds for all $c_1,c_2\in \CC$.
Next we \checks{compute what} is $\Sol$ in the isomorphism.
First, $\Sol$ is a $2$-dimensional module whose basis is $\{1-x\e,\e\}$.
Second, the quiver representation of $\Sol$ is
\begin{center}
\begin{tikzpicture}
\draw (0.2,0) node[ left]{$(1-x\e)\CC\oplus\e\CC$};
\draw (2,0)   node[right]{$\e\cdot(-)$,};
\draw[->][line width=0.75pt][rotate around={20:(0.5,0)}] (0,0) arc(-180:150:1);
\end{tikzpicture}
\end{center}
where $\e\cdot(-): (1-x\e)\CC\oplus\e\CC \to (1-x\e)\CC\oplus\e\CC$ is a $\CC$-linear map whose image is
$\im(\e\cdot(-)) = \CC \e \subseteq (1-x\e)\CC$, i.e., $\e\cdot(-) \ne 0$.
By \checks{the} Auslander--Reiten quiver
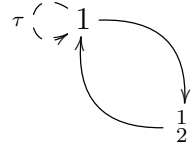
\[ \xymatrix{
  {\small 1} \ar@/^1.5pc/[rd]  \ar@{-->}@(lu,ld)_{\tau} & \\
   & {\begin{smallmatrix}1\\1\end{smallmatrix}}  \ar@/^1.5pc/[lu]
}  \]
of $\itLamb$, we obtain that $\Sol \cong {_{\itLamb}\left(\begin{smallmatrix}1\\1\end{smallmatrix}\right)}$
is the indecomposable projective module corresponding the vertex $1$ of the quiver $\Q$, i.e.,
\[ \Sol \cong {_{\itLamb}\left(\begin{smallmatrix}1\\1\end{smallmatrix}\right)} \cong \itLamb = V \]
in this example.
\end{example}

\paragraph{Authors' Contributions}
The order of authors is alphabetical, and all authors contributed equally to the conception, methodology, derivation, and writing of this paper.

\paragraph{Competing Interests}
he authors declare that they have no conflicts of interest as defined by the journal, nor any other interests that could be perceived as influencing the results presented in this paper.

\paragraph{Data availability}
Data sharing not applicable to this article as no datasets were generated or analysed during the current study

\paragraph{Ethical Approval}
This article does not require ethical approval.

\paragraph{Fundings}
Yu-Zhe Liu is supported by
the National Natural Science Foundation of China (Grant Nos. 12401042 and 12561008),
the Science and Technology Foundation of the Guizhou S\&T Department (Grant Nos. VZD[2026]001, ZD[2025]085 and ZK[2024]YiBan066),
and Scientific Research Foundation of Guizhou University (Grant No. [2023]16).
Shengda Liu is supported by the National Natural Science Foundation of China (Grant No. 62203442) and
the Beijing Natural Science Foundation (Grant No. L2603026).

\paragraph{Acknowledgements}
We are grateful to the reviewer for the constructive comments, which have improved the clarity and consistency of the paper.



\addcontentsline{toc}{section}{References}
%
%
%

\def\cprime{$'$}

\end{document}